\documentclass[reqno]{amsart}
\usepackage[left=1in,right=1in,top=1in,bottom=1in]{geometry}
\usepackage{tikz}
\usetikzlibrary{shapes,decorations,calc,arrows}
\usepackage[foot]{amsaddr}
\usepackage{amsthm}
\usepackage{amscd}
\usepackage{amsfonts}
\usepackage{amsmath}
\usepackage{amssymb}
\usepackage{mathrsfs}
\usepackage{multirow}
\usepackage{verbatim}
\usepackage{url}
\usepackage[hidelinks]{hyperref}
\usepackage{graphicx}
\usepackage{cite}
\usepackage{fancyhdr}
\usepackage{yfonts}
\usepackage{setspace}
\usepackage{titlesec}
\usepackage{enumitem}
\usepackage{multicol}
\usepackage{relsize} 
\usepackage{ dsfont }
\usepackage{tikz-cd}

\usepackage{tocloft} % fixes some weird TOC spacing

\titleformat{\section}[hang]
{\normalfont\Large\bfseries}
{\thesection.}{0.5em}{}

\titlespacing*{\section}{0pc}{2pc}{0.25pc}

\titleformat{\subsection}[runin]
{\normalfont\large\bfseries}
{\thesubsection}{0.5em}{}

\titlespacing{\subsection}{0pc}{1.5pc}{0.5pc}

\DeclareMathOperator{\ev}{ev}
\DeclareMathOperator{\dom}{dom}

\newcommand{\<}{\left\langle}
\renewcommand{\>}{\right\rangle}

\newcommand{\C}{\mathbb{C}}

\newcommand{\N}{\mathbb{N}}

\newcommand{\cum}{\kappa_B}
\newcommand{\pe}{\partial_\eta}
\newcommand{\pea}{\partial_\eta^\ast}
\newcommand{\pej}{\partial_{\eta,j}}
\newcommand{\peaj}{\partial_{\eta, j}^\ast}
\newcommand{\bx}{B{\<\mathbf{x}\>}}
\newcommand{\x}{\mathbf{x}}
\newcommand{\con}{E_{B'\cap M}}

\newcommand{\at}{\tilde{A}}

\newcommand{\cpe}{\overline{\pe}}
\newcommand{\xit}{\tilde{x}_i}
\newcommand{\ait}{\tilde{A}_i}

\newtheorem{thm}{Theorem}[section]
\newtheorem{prop}[thm]{Proposition}
\newtheorem{lem}[thm]{Lemma}
\newtheorem*{lem*}{Lemma}
\newtheorem{cor}[thm]{Corollary}

\theoremstyle{definition}
\newtheorem{defi}[thm]{Definition}
\newtheorem{ex}[thm]{Example}

\newtheorem{rem}[thm]{Remark}

\newtheorem{thmalpha}{Theorem}

\title{\textbf{ON CONJUGATE SYSTEMS WITH RESPECT TO COMPLETELY POSITIVE MAPS}}
\author{Yoonkyeong Lee}
\address{Department of Mathematics, Michigan State University\hfill \url{leeyoo16@msu.edu}}
\date{}

\begin{document}

\maketitle
%\tableofcontents
\begin{abstract}
We study the operator-valued partial derivative associated with covariance matrices on a von Neumann algebra $B$. We provide a cumulant characterization for the existence of conjugate variables and study some structure implications of their existence. Namely, we show that the center of the von Neumann algebra generated by $B$ and its relative commutant is the center of $B$.
\end{abstract}

%%%%%%%%%%%%%%%%%%%%%%%%%%%%%
%       Introduction        %
%%%%%%%%%%%%%%%%%%%%%%%%%%%%%

\section*{Introduction}

The free Fisher information and free entropy introduced by Voiculescu are analogues of these quantities in classical probability theory and have been useful tools in the study of von Neumann algebras. In particular, using a nonmicrostates approach, Dabrowski showed in \cite{MR2606868} that $W^\ast(x_1,...,x_n)$ is factor  when the free Fisher information $\Phi^*(x_1,\ldots, x_n)$ is finite. Operator-valued free probability has also been a topic of interest in the last three decades since Voiculescu introduced operator-valued free product (See \cite{MR799593, MR1372537, MR1407898, MR1704661, MR1737257, MR2159789, MR4198974, Ito24}). The goal of this paper is to extend Dabrowski's results to the operator-valued setting, where the coefficient algebra $\C$ is replaced with a von Neumann subalgebra $B$. 

Let $(M,\tau)$ be a tracial von Neumann algebra with a von Neumann subalgebra $B\leq M$ and assume that $M$ is generated by $B$ and a tuple of self-adjoint operators $\x=(x_i)_{i\in I}$. For a covariance matrix $\eta$ on $B$, we extend it to $M$ by composing with $E_B$, a conditional expectation onto $B$, and still denote $\eta:=\eta\circ E_B$. Then we consider a derivation $\pe$ on $\bx$ valued in a correspondence $L^2(M\boxtimes_\eta M,\tau)$ over $M$, which is determined by $\eta$ (see Definition \ref{defi:partial derivative}). This derivation is a multivariate version of the one considered by Shlyakhtenko in \cite{MR1737257}. A $(B,\eta)$-conjugate system is then a tuple of vectors $(\xi_i)_{i\in I} \subset L^2(M,\tau)$ which implement an integration-by-parts formula relative to $\pe$ (see Definition \ref{defi: conj var}). When $\eta$ is a single map along the diagonal, the existence of a conjugate system is equivalent to $\Phi^\ast(x_1,...,x_n:B,\eta)< \infty $ from \cite{MR1737257}.
We utilize a combination of combinatorial techniques of Mai, Speicher, and Weber \cite{MR3558227}, Popa's intertwining theorem \cite[Theorem 5.1]{MR2334200}, and Dabrowski's derivation methods \cite[Theorem 1]{MR2606868} to show the following:

\begin{thmalpha} [{Theorem~\ref{thm:main}}]\label{thm a}
    Let $(M,\tau)$ be a tracial von Neumann algebra generated by a von Neumann algebra $B$ and a tuple of self-adjoint operators $\x=(x_i)_{i\in I}$ with $|I|>1$. Assume that a $(B,\eta)$-conjugate system exists for $\x$ and a covariance matrix $\eta=(\delta_{ij}E_B)_{i,j\in I}$. Then \[Z(B\vee(B'\cap M))=Z(B).\] In particular, $Z(M)\subset Z(B)$.     
\end{thmalpha}

As an immediate consequence of Theorem~\ref{thm a}, one has $Z(M)\subset Z(B)$. This recovers Dabrowski's factoriality result from \cite[Theorem 1]{MR2606868} when $B=\C$. This also recaptures $Z(M)\subset Z(B)$ for a von Neumann algebra $M$ generated by $B$ and a $B$-valued semicircular family (see \cite[Example 3.2, 3.3(b)]{MR1704661}). Thus Theorem~\ref{thm a} can be viewed as a generalization to tuples $\x$ that are close to a $B$-valued semicircular family in the sense that they admit a $(B,\eta)$-conjugate system. One another consequence of Theorem~\ref{thm a} is that if $B$ is a factor, then every intermediate algebra $B\vee(B'\cap M)\leq N \leq M$ is irreducible in $M$ (see Corollary~\ref{cor:irred}).
To prove the above theorem, we first show the von Neumann algebras $B\<E_{B'\cap M}(x_i)\>''$ do not intertwine into $B$ inside of $B\vee(B'\cap M)$. This implies, in particular, that there are no $B\<E_{B'\cap M}(x_i)\>''$-central vectors in $L^2(B\vee(B'\cap M)\boxtimes_\eta B\vee(B'\cap M),\tau)$, which then allows us to implement the same strategy used by Dabrowski. 

Our next theorem is critical in showing the lack of intertwining mentioned above. We use an argument similar to Mai, Speicher and Weber, which relies on the closability of $\pe$ when $(B,\eta)$-conjugate system exists. Recall that we say a self-adjoint element $x$ has no atoms in $B$ when $\ker(x-b)=0$ for all self-adjoint $b\in B$ (see \cite[Section 2]{MR4198974}).

\begin{thmalpha}[{Theorem~\ref{thm:no atom}}] \label{thm b}
    Let $(M,\tau)$ be a tracial von Neumann algebra generated by a von Neumann algebra $B$ and a tuple of self-adjoint operators $\x=(x_i)_{i\in I}$. Assume that a $(B,\eta)$-conjugate system exists for $\x$ and a covariance matrix $\eta$ with $\eta_{ii}=E_B$ for all $i\in I$. Then a self-adjoint $B$-linear combination $P=\sum_{j}a_jx_ib_j +b_0$ has no atoms in $B$ when $a_j,b_j\in B$ and there exists a positive scalar $c$ such that $\sum_{j}a_jb_j\geq c\cdot 1$.

    If we furthermore assume $\eta=(\delta_{ij}E_B)_{i,j\in I}$, then a self-adjoint $B$-linear combination $P=\sum_{i}\sum_{j}a_j^{(i)}x_ib_j^{(i)} +b_0$ has no atoms in $B$ when $a_j^{(i)},b_j^{(i)}\in B$ and there exists an $1\leq i\leq n$ and a positive scalar $c$ such that $\sum_{j}a_j^{(i)}b_j^{(i)}\geq c\cdot 1$.
\end{thmalpha}

We additionally obtain a characterization for $(B,\eta)$-conjugate systems in terms of $B$-valued cumulants ($\cum^{(d)})_{d\in \N}$. It is often easier to work with cumulants rather than directly with moments---particularly in the operator-valued case---and in our setting it allows us to more easily confirm examples such as $B$-valued semicircular systems (see Examples~\ref{ex:semicircular}, \ref{ex:amalgamation}). A $B$-valued cumulant characterization for conjugate variable has been obtained in \cite[Theorem 4.1]{MR1907203} when $\eta=\delta_{ij}\tau$ and in \cite[Lemma 4]{MR2159789} when $|I|=1$ and $\eta=\delta_{ij}E_B$. For general $\eta$, we show the following:

\begin{thmalpha}[{Theorem~\ref{thm:cumulants}}] \label{thm c}
Let $(M,\tau)$ be a tracial von Neumann algebra generated by a von Neumann algebra $B$ and a tuple of self-adjoint operators $\x=(x_i)_{i\in I}$. Then $\xi_i\in L^2(M,\tau)$ is the $i$-th $(B,\eta)$-conjugate variable for $\x$ if and only if for all $b,b_1,\ldots, b_d\in B$ and $i_1,\ldots, i_d\in I$ one has
     \[\begin{cases} 
        \cum^{(1)}(\xi_i b) = 0,  \\
        \cum^{(2)}(\xi_i \otimes bx_j)=\eta_{ij}(b), \\ 
        \cum^{(d+1)}(\xi_i \otimes b_1x_{i_1} \otimes \cdots \otimes b_dx_{i_d})=0 ,d\geq 2
    \end{cases}.
    \]
\end{thmalpha}
\noindent

The paper is organized as follows. In Section~\ref{section: preliminaries}, we review several topics necessary for this paper including correspondences, operator-valued semicircular operators, operator-valued combinatorics, and Popa's intertwining theorem. In Section~\ref{section: conjugate system}, we define the $\eta$-partial derivative with respect to a completely positive map $\eta:B\to B\otimes B(\ell^2(I))$, and we define $(B,\eta)$-conjugate systems. There we also prove Theorem~\ref{thm c} and present some examples. In Section~\ref{section: no atoms}, we establish some technical bounds for $\eta$-partial derivatives and prove Theorem~\ref{thm b}. Section~\ref{section: relative diffuse} is devoted to the proof of Theorem~\ref{thm a}.

\subsection*{Acknowledgements}
I would like to express my gratitude to my advisor, Brent Nelson, for his guidance, and encouragement throughout the course of this research. I am particularly grateful for the original idea he provided, which helped to advance this paper. I would like to thank Ionut Chifan, Adrian Ioana, Tobias Mai, Gregory Patchell, Aldo Garcia Guinto and Srivatsav Kunnawalkam Elayavalli for their invaluable expertise and discussion which greatly contributed to the development of this article. This research was supported by NSF grant DMS-2247047.

%%%%%%%%%%%%%%%%%%%%%%%%%%%%%
%       Preliminaries       %
%%%%%%%%%%%%%%%%%%%%%%%%%%%%%

\section{Preliminaries}\label{section: preliminaries}
Throughout the paper, $M$ denotes a von Neumann algebra equipped with a faithful normal tracial state $\tau$, and we call the pair $(M,\tau)$ a \emph{tracial von Neumann algebra}. We will use lattice notation: $B\leq M$ denotes that $B$ is a von Neumann subalgebra of $M$  and $B_1\vee B_2$ denotes the von Neumann algebra generated by two von Neumann subalgebras $B_1, B_2\leq M$. We denote the center of a von Neumann algebra by $Z(M):=M'\cap M$. We say $B$ is irreducible in $M$ if $B'\cap M=\C$. The normalizer of $B$ in $M$ is $\mathcal{N}_M(B):=\{u\in M: u \text{ unitary }, uBu^\ast=B\}$.
$L^2(M,\tau)$ denotes the Hilbert space from the GNS construction associated to $\tau$ with $\<a,b\>_\tau=\tau(a^\ast b)$, the norm induced by $\tau$ on $L^2(M,\tau)$ is denoted by $\| \cdot  \|_\tau$ and $J_\tau$ denotes the canonical conjugation operator on $L^2(M,\tau)$ which is determined by $J_\tau(x)=x^\ast$ for $x\in M$.

Given two von Neumann algebras $N,M$, a \emph{$N,M$-correspondence} is a Hilbert space $H$ equipped with two normal unital $\ast$-homomorphisms 
    \[ \lambda:N\to B(H), \quad  \rho: M^{op} \to B(H)\]
such that $\lambda(x)\rho(y)=\rho(y)\lambda(x)$ for all $x\in N, y\in M$. We write $x\xi y:=\lambda(x)\rho(y)\xi$ for $\xi\in H$. When $N=M$, $H$ is called a (\emph{von Neumann}) \emph{correspondence} over $M$. For a correspondence $H$ over $M$ and a von Neumann subalgebra $B\leq M$, we say $\xi\in H$ is a $B$-\emph{central vector} if $b\xi=\xi b$ for all $b\in B$.

Let $A$ be a $C^{\ast}$-algebra. An \emph{inner-product $A$-module} is a linear space $E$ which is a $A$-module together with a map $E\times E \to A : (x,y) \to \< x \mid y\>_A$ such that 
    \begin{enumerate}
        \item $\< x \mid \alpha y +\beta z \>_A = \alpha \< x \mid y\>_A +\beta\< x \mid z\>_A$ 
        \item $\< x \mid ya\>_A = \< x \mid y \>_Aa$
        \item $\< y\mid x\>_A= \< x \mid y\>_A^\ast$
        \item $\< x\mid x\>_A \geq 0$ with equality if and only if $x=0$
    \end{enumerate}
for $x,y\in E$, $a\in A$ and $\alpha,\beta\in \C$.
If $E$ is an inner product $A$-module, for $x\in E$, $\| x\| =\| \< x\mid x \>_A\| ^{1/2}$ is a norm on $E$. We say an inner-product $A$-module is a \emph{Hilbert $A$-module} if it is complete with respect to its norm. We also have a module version of the Cauchy-Schwarz inequality: 
    \[|\<x\mid y\>_A|^2 \leq \|\<x\mid x\>_A\|\<y\mid y\>_A.\] 
For a Hilbert $A$-module $E$, define the set of adjointable maps $\mathcal{L}(E)$ to be the set of all maps $t:E\to E$ for which there is a map $t^\ast:E\to E$ such that $\<tx\mid y\>_A=\<x \mid t^\ast y\>_A$ for all $x,y\in E$. Such maps are automatically $A$-linear and bounded (see \cite{MR1325694}).
A \emph{$C^\ast$-correspondence} over $A$ is Hilbert $A$-module $E$ along with a $\ast$-homomorphism $\varphi_E:A\to \mathcal{L}(E)$. We refer to $\varphi_E$ as the left action of a $C^\ast$-correspondence $E$ which we will simply write $ax$ for $\varphi_E(a)x$. For further details, see \cite{MR2102572} and \cite{MR1325694}.

%%%%%%%%%%%%%%%%%%%%%%%%%%%%%%%%
%%%%%%%%%%%%%%%%%%%%%%%%%%%%%%%%
\subsection{B-valued semicircular operators}\label{subsec:B-valued semicircular} We recall the notion of an $B$-valued semicircular system from \cite[Section 2]{MR1704661}. Let $B$ be a von Neumann algebra with a normal faithful tracial state $\tau$ and $\eta_{ij}:B\to B$ linear maps with $i,j$ in a countable index set $I$. Let $\{e_{i,j}\in B(\ell^2(I))\colon i,j\in I\}$ be a family of matrix units and assume that $\eta(b):= \sum_{i,j\in I} \eta_{ij}(b)\otimes e_{ij}$ converges in the $\sigma$-weak operator topology for all $b\in B$, so that $\eta$ is a map from $B$ to $B\otimes B(\ell^2(I))$. If $\eta$ is normal and completely positive, it is called a \emph{covariance matrix}. By \cite[Lemma 2.2]{MR1704661}, there exists a unique $C^\ast$-correspondence over $B$, which we denote by $B\boxtimes_\eta B$, equipped with vectors $\{e_i:i\in I\}\subset B\boxtimes_\eta B$ with property $\<ae_ib\mid ce_jd\>_B=b\eta_{ij}(a^\ast c)d$ for $a,b,c,d\in B$ and $B\boxtimes_\eta B$ is the closed linear span of $\{ae_ib:a,b\in B, i\in I\}$. Define $L^2(B\boxtimes_\eta B,\tau)$ to be the von Neumann correspondence over $B$ which is the closure of $B\boxtimes_\eta B$ with respect to the inner product 
     \[\< ae_ib ,ce_jd\>=\tau\left(\< ae_ib\mid ce_jd\>_B\right)=\tau\left(b^\ast \eta_{ij}(a^\ast c)d\right).\]
Note that $B\boxtimes_\eta B$ and $L^2(B\boxtimes_\eta B,\tau)$ admit the natural left and right actions of $B$ and the latter are normal.
Consider the full Fock space 
    $$\mathcal{F}=B\oplus \bigoplus_{n>0}(B\boxtimes_\eta B)^{\otimes_B^n},$$
and $B$-linear operators $L(e_i)$ on $\mathcal{F}$ defined by         
    \[L(e_i)b=e_i\cdot b,\] \[L(e_i)\cdot x_1\otimes \cdots \otimes x_d=e_i\otimes x_1\otimes \cdots \otimes x_d.\]
Denote its completion with respect to the inner product $\<\cdot,\cdot \>=\tau(\<\cdot \mid \cdot \>_B)$ by $\mathcal{F}_\tau$, and let $L_\tau(e_i)$ denote the unique bounded extension of $L(e_i)$ to $\mathcal{F}_\tau$.
The family of operators $s_i:=L_\tau(e_i)+L_\tau(e_i)^\ast$ for $i\in I$ is called a \emph{$B$-valued semicircular family with covariance $\eta$}, and one denotes $\Phi(B, \eta):= (B\cup \{s_i:i\in I\})'' \leq B(\mathcal{F}_\tau)$. Note that $\Phi(B, \eta)$ does not depend on $\tau$ by \cite[Proposition 2.15]{MR1704661}. Recall from \cite[Lemma 2.10]{MR1704661} that there exists a normal conditional expectation $E_B:\Phi(B,\eta)\to B$ and from \cite[Proposition 2.20]{MR1704661} that  if $\eta$ is $\tau$-symmetric, i.e., $\tau(\eta_{ij}(a)b)=\tau(a\eta_{ji}(b))$ holds for all $i,j \in I$ and $a,b\in B$, then $\tau\circ E_B$ is tracial on $\Phi(B,\eta)$. Moreover, in this case $E_B$ is faithful by \cite[Proposition 5.2]{MR1704661}.

\subsection{Operator-valued combinatorics}\label{section: op cum} Combinatorics over the lattice of non-crossing partitions has provided many insights to free probability theory, particularly through cumulants which provide a way to ``linearize'' free independence. In \cite{MR1407898}, Speicher generalized this combinatorial approach to the operator-valued case. Let $NC(d)$ denote the set of all non-crossing partitions of $\{1,...,d\}$. Let $B$ be a unital algebra and $A$ a $B$-$B$-bimodule. $A^{\otimes_B^d}$ denotes the $d$-fold $B$-tensor product of $A$ with itself which is also a $B$-$B$-bimodule over $B$. That is, $A^{\otimes_B ^d}$ is spanned by elements of the form: \[(a_1\otimes_B \cdots \otimes a_k\otimes ba_{k+1} \otimes_B\cdots \otimes a_d) = (a_1\otimes_B \cdots \otimes_B a_kb \otimes_B a_{k+1}\otimes_B \cdots \otimes a_n)\] for $a_1,...,a_d\in A$ and $b\in B$. I will suppress the subscript $B$ from the notation $\otimes_B$ and use $\otimes$ for convenience.

For each $d\in \N$, let $f^{(d)}:A^{\otimes_B^d}\to B$ be a linear $B$-$
B$-bimodule map. Then define the corresponding \emph{$B$-valued multiplicative function} 
\begin{align*}
    \hat{f} : \bigcup_{d=1}^\infty(NC(d)\times A^{\otimes_B^d})&\to B\\
    (\pi,a_1\otimes \cdots \otimes a_d) &\mapsto \hat{f}(\pi)[a_1\otimes \cdots\otimes a_d]
\end{align*}
where $\hat{f}(\pi)[a_1\otimes\cdots\otimes a_d]$ is defined recursively as follows: if $\pi=\pi_1\cup 1_{[k,l]}$ then
    \[ \hat{f}(\pi)[a_1\otimes\cdots\otimes a_d] =\hat{f}(\pi_1)[a_1\otimes \cdots\otimes a_{k-1} f^{(l-k+1)}(a_k\otimes\cdots\otimes a_l)\otimes a_{l+1}\otimes \cdots \otimes a_d]\] 
    with the base case $\hat{f}(\emptyset)(b)=b$ for $b\in B$.
For example, when $\pi = \{1,4\} \cup \{2,3\}$, $\hat{f}(\pi)[a_1\otimes a_2\otimes a_3\otimes a_4] = f^{(2)}(a_1f^{(2)}(a_2\otimes a_3)\otimes a_4)$.
Two important examples of multiplicative maps come from conditional expectations and their associated cumulants.

Let $B\leq M$ be an inclusion of von Neumann algebras and let $E_B:M\to B$ be faithful normal conditional expectation. Then viewing $M$ as a $B$-$B$-bimodule, the $B$-valued multiplicative function $\hat{E}_B=(E_B^{(d)}:M^{\otimes_B d} \to B)_{d\in\N}$ defined by  \[E_B^{(d)}(a_1\otimes a_2\otimes \cdots \otimes a_d) = E_B(a_1a_2\cdots a_d)\] is called the \emph{moment function}.
The \emph{cumulant function} with respect to $E_B$ is a multiplicative function $\hat{\kappa}_B=(\cum^{(d)}:M^{\otimes_B d} \to B)_{d\in\N}$ is defined by the moment-cumulant formulas (see more details in \cite[Example 1.2.2 and Proposition 3.2.3]{MR1407898}):
\begin{align*}
    E_B(a_1\cdots a_d)&=\sum_{\pi \in NC(d)} \cum(\pi)[a_1\otimes a_2\otimes \cdots \otimes a_d]\nonumber
\end{align*}
which can be written as
\begin{align}\label{pre:moment cumulant d}   
    E_B(a_1\cdots a_d)
    &=\cum^{(d)}(a_1\otimes \cdots \otimes a_d) +\sum_{\substack{\pi\in NC(d)\\ \pi \not\equiv 1_d}}\cum(\pi)[a_1\otimes\cdots \otimes a_d].
\end{align}
where $\pi\equiv 1_d$ denotes $\pi$ being only one block of $\{1,...,d\}$.

As an illustration, when $d=2$, we have
\begin{align}\label{pre:cumulant 2}
    E_B(a_1a_2) =\cum(\{1,2\})[a_1\otimes a_2] +\cum(\{1\}\cup\{2\})[a_1\otimes a_2]=\cum^{(2)}(a_1\otimes a_2) +\cum^{(1)}(a_1)\cum^{(1)}( a_2). 
\end{align}
When $d=3$, it is easy to check
\begin{align*}
    \cum^{(3)}(a_1\otimes a_2\otimes a_3)=E_B(a_1a_2a_3)&-E_B(a_1)E_B(a_2a_3)-E_B(a_1E_B(a_2)a_3)\nonumber\\&-E_B(a_1a_2)E_B(a_3)+2E_B(a_1)E_B(a_2)E_B(a_3).
\end{align*}
Note that $\cum^{(d)}(a_1\otimes \cdots \otimes a_d)=0$ if $a_i\in B$ for some $i\in \{1,..,d\}$ for $d>1$. Recall from \cite[Section 3.3]{MR1407898} that for a family of subalgebras $(A_i)_{i\in I}$ of $M$ with $B\subset A_i$ for all $i\in I$, we say the subalgebras $(A_i)_{i\in I}$ are \emph{free with amalgamation over $B$} if $E_B(a_1\cdots a_d)=0$ whenever $a_i\in A_{j_i}, j_1\neq j_2\neq  \cdots \neq j_d$ and $E_B(a_i)=0$ for all $j_i\in I$. Equivalently, for all $d\geq 2$ and $a_i\in A_{j_i}$ for $i=1,\dots,d$ and $j_i\in I$ one has $\cum^{(d)}(a_1\otimes \cdots \otimes a_d)=0$ whenever $i_1,...,i_d$ are not all equal.

\subsection{Popa's intertwining theorem}
     In case of $B=\C$, Dabrowski's proof in \cite[Theorem 1]{MR2606868} required diffuseness of $\C\<x_i\>''$ to argue that there is no vectors in $L^2(M\bar\otimes M)$ that are $x_i$-central. Similarly, in operator-valued case we need a tool to argue absence of central vectors. 
\begin{thm}\cite[Theorem 5.1]{MR2334200} \label{pre:popathm}  Assume a von Neumann algebra $M$ has a faithful normal tracial state $\tau$ and $A,B\leq M$ are unital von Neumann subalgebras. Then the following are equivalent:
    \begin{enumerate}
        \item There exist projections $p\in P(A)$, $q\in P(B)$, a nonzero partial isometry $v\in qMp$ and a unital normal $\ast$-homomorphism $\theta:pAp\to qBq$ satisfying $$\theta(pap)v=v(pap).$$ for all $a\in A$.
        \item There does not exist a net of unitaries $(u_i)_{i\in I}\subset A$ such that $\lim_{i\to\infty}\| E_B(xu_iy)\| _\tau=0$ for any $x,y\in M$.
        \item\label{popa 3} There exists a $A$-$B$-correspondence $K\leq L^2(M,\tau)$ such that $\dim_B(K)<\infty$.
    \end{enumerate}
\end{thm}

If either of the statements mentioned in the above Theorem holds, we say \emph{a corner of $A$ embeds into $B$} inside $M$ and write $A\prec_M B$. We say $M$ is \emph{diffuse relative} to $B$ if $M\nprec_M B$.  
Using Theorem~\ref{pre:popathm}.(\ref{popa 3}), one can show the following lemma which we will use in Section~\ref{section: relative diffuse}.
\begin{lem} \label{pre:lem1}
    Assume $M$ has a faithful normal tracial state and $B\leq A \leq M$ are unital von Neumann subalgebras. Then $M\prec_M B$ implies $A\prec_A B$.
\end{lem}

%%%%%%%%%%%%%%%%%%%%%%%%%%%%%
%      SECTION 2       %
%%%%%%%%%%%%%%%%%%%%%%%%%%%%%
\section{$(B,\eta)$-conjugate systems}\label{section: conjugate system}
Let $(M,\tau)$ be a tracial von Neumann algebra, let $B\leq M$ be a unital von Neumann subalgebra, and let $\x=(x_i)_{i\in I}\subset M$ be a tuple of self-adjoint operators indexed by a countable set $I$. We write $\bx$ for the $*$-algebra generated by $B$ and the set $\{x_i\colon i\in I\}$, and we will assume that $M=\bx''$. Since $M$ has a faithful normal tracial state, there exists a unique normal faithful conditional expectation $E_B:M\to B$ satisfying $\tau=\tau \circ E_B$. 
Let $\{e_{i,j}\in B(\ell^2(I))\colon i,j\in I\}$ be a family of matrix units and $\eta:B\to B\otimes B(\ell^2(I))$ be a covariance matrix on $B$ given by 
    \[\eta(b):= \sum_{i,j\in I} \eta_{ij}(b)\otimes e_{ij}.\]
We will still denote $\eta:=\eta \circ E_B :M\to B\otimes B(\ell^2(I)) \subset M\otimes B(\ell^2(I))$ and we assume $\eta$ is $\tau$-symmetric.
Recall from the Section~\ref{subsec:B-valued semicircular}, we define a $C^\ast$-correspondence $M\boxtimes_\eta M$ over $M$ and a von Neumann correspondence $L^2(M\boxtimes_\eta M,\tau)$ over $M$ with respect to $\eta$. The inner product on $L^2(M\boxtimes_\eta M,\tau)$ is 
    \[\< ae_ib ,ce_jd\>=\tau\left(\< ae_ib\mid ce_jd\>_M\right)=\tau\left(b^\ast \eta_{ij}(a^\ast c)d\right).\]
The norm induced by the above inner product is denoted by $\|\cdot \|_2$.

Now we define a partial derivative on $\bx$ with respect to a completely positive map $\eta$. Note that this previously appeared in work of Shlyakhtenko \cite{MR1737257} in the one variable case.

\begin{defi}\label{defi:partial derivative}
The \textbf{$\eta$-partial derivative} $\partial_\eta:\bx\to L^2(M\boxtimes_\eta M,\tau)$ is the linear mapping that satisfying
    \begin{align*}
        \pe(x_i) &= e_i \qquad i\in I,\\
        \pe(b) &=0 \qquad b\in B,
    \end{align*}
and the Leibniz rule $\partial_\eta(pq) = \partial_\eta(p)\cdot q+ p\cdot \partial_\eta(q)$ for $p,q\in\bx$. That is,
        \[ \pe(b_0x_{i_1}b_1\cdots x_{i_d}b_d)=\sum_{k=1}^d b_0x_{i_1}\cdots b_{k-1}e_{i_k}b_k\cdots x_{i_d}b_d. \]
\end{defi}

Note that $\pe$ may not be well-defined in general due to algebraic relations between $B$ and $\x$. One way to avoid this difficulty is to simply assume that $B$ and $\x$ are algebraically free, but this excludes $B$-valued semicircular operators which commute with $B$ (see \cite[Lemma 7.1]{Ito24}). Alternatively, Shlyakhtenko offers another approach that extends $M$ to a larger algebra where $\x$ and $B$ are algebraically free on \cite[Lemma 3.2]{MR1737257}. Another way is to define $\pe$ on abstract polynomials with coefficients in $B$ and then evaluate in the tuple $\x$ after, but it turns out that in the situation we will consider $\pe$ is always well defined (see Remark~\ref{rem:der well def}). 

With this derivative, we have a generalized definition of the conjugate variables from \cite[Definition 3.3]{MR1737257}.

\begin{defi}\label{defi: conj var}
    We say $\xi_i \in L^2(M,\tau)$ is the $i$-th $(B,\eta)$\textbf{-conjugate variable} for $\x$ if it satisfies
        \begin{align}\label{eqn:integration_by_parts}
            \< \xi_i , p \>_\tau = \< e_i ,\partial_\eta(p)\> \qquad p\in \bx.
        \end{align}
    More explicitly, for all $j_1,\ldots, j_d\in I$ and $b_0,b_1,\ldots, b_d\in B$ one has
        \[ \<\xi_i,b_0x_{j_1}b_1\cdots x_{j_d}b_d\>_\tau=\sum_{k=1}^d \tau\left(\eta_{ij_k}(b_0x_{j_1}\cdots b_{k-1})E_B\left(b_k\cdots x_{j_d}b_d\right)\right).\] 
    If such elements $\xi_i$ exist for each $i\in I$, we say the tuple $(\xi_i)_{i\in I}$ is a \textbf{$(B,\eta)$-conjugate system} for $\mathbf{x}$.
\end{defi}

In the case of $B=\C$ and $\eta_{ij}=\tau\delta_{ij}$, a $(B,\eta)$-conjugate system corresponds to the conjugate system for $\mathbf{x}$ introduced by Voiculescu in \cite[Definition 3.1]{Voi98}.
Note that (\ref{eqn:integration_by_parts}) uniquely determines $\xi_i$ because of the density of $\bx$ in $L^2(M,\tau)$, and it is equivalent to saying $\xi_i=\pea(e_i)$. Thus $e_i$ is in the domain of $\pea$ when the $i$-th $(B,\eta)$-conjugate variable exists, and, in fact, we will later see that $\bx\cdot e_i \cdot \bx \subset \dom(\pea)$ (see Proposition~\ref{prop:adjoint} below). 

We also remark that the $(B,\eta)$-conjugate variables are self-adjoint in the sense that $J_\tau\xi_i=\xi_i$. To see this, we define $J:L^2(M\boxtimes_\eta M,\tau) \to L^2(M\boxtimes_\eta M,\tau)$ by $J(ae_ib)=b^\ast e_i a^\ast$. Then $J$ is conjugate linear and the property $\tau(\eta_{ij}(a)b)=\tau(a\eta_{ji}(b))$ ensures that $J$ is an isometry, in particular $J$ is well-defined. One also has $J\pe(p) = \pe(p^*)$ for $p\in \bx$, which can be seen directly from the definition of the $\eta$-partial derivative. Thus we have 
    \[ \<p,J_\tau \xi_i\>_\tau =\<\xi_i,p^\ast\>_\tau = \<e_i, \pe(p^\ast)\> = \<e_i, J\pe(p)\>=\<\pe(p), Je_i\> =\< \pe(p),e_i\> = \<p,\xi_i\>_\tau\]
for all $p\in \bx$.

\begin{rem}\label{rem:der well def}
Consider the algebra $B\<\mathbf{t}\>$ as the (algebraic) free product of $B$ and the algebra of non-commutative polynomials in abstract self-adjoint variables $t_i$ for $i\in I$. One may consider a derivation $\tilde{\partial}_\eta$ from $B\<\mathbf{t}\>$ to $\bigoplus_{i\in I} B\<\mathbf{t}\>\otimes B\<\mathbf{t}\>$ determined by $\tilde{\partial}_\eta(b)=0$, $\tilde{\partial}_\eta(t_i)=(1\otimes 1)_i$, linearity, and the Leibniz rule. 
Define an evaluation map $\ev_\eta:\bigoplus_{i\in I} B\<\mathbf{t}\>\otimes B\<\mathbf{t}\>\to  M\boxtimes_\eta M  $ by $\ev_\eta((a\otimes b)_i) = \ev_\x(a)e_i \ev_\x(b)$ where the evaluation map $\ev_\x\colon B\<\mathbf{t}\>\to \bx$ is given by mapping $t_i\to x_i$.
    \[
    \begin{tikzcd}[row sep=large, column sep=large] 
    B\<\mathbf{t}\> \arrow{r}{\tilde{\partial}_\eta} \arrow[swap]{d}{\ev_\x} & \displaystyle\bigoplus_{i\in I}B\<\mathbf{t}\>\otimes B\<\mathbf{t}\> \arrow{d}{\ev_\eta} \\
B\<\x\>  & M\boxtimes_\eta M
    \end{tikzcd}
    \] 
Then one could define the $i$-th $(B,\eta)$-conjugate system $\xi_i\in L^2(M,\tau)$ satisfying:
    \[ \< \xi_i, \ev_\x(P)\>_\tau = \< e_i, \ev_\eta\circ\tilde{\partial}_\eta(P)\>\]
for all $P\in B\<\mathbf{t}\>$.
If a $(B,\eta)$-conjugate system exists, then $\ker(\ev_\x)\subset \ker(\ev_\eta\circ\tilde{\partial}_\eta)$ by the following computation: if $\ev_\x(P(\mathbf{t}))=0$, then for any $P_1,P_2\in B\<\mathbf{t}\>$ and $i\in I$ one has
\begin{align*}
    0&=\<\ev_\x(P_1)\ev_\x(P)\ev_\x(P_2),\xi_i  \>_\tau = \< \ev_\x(P_1PP_2),\xi_i\>_\tau \\
    &= \<(\ev_\eta\circ\tilde{\partial}_\eta)(P_1PP_2),e_i\>\\
    &=\< \ev_\eta\left( \tilde{\partial}_\eta(P_1)PP_2 +P_1\tilde{\partial}_\eta(P)P_2 +P_1P\tilde{\partial}_\eta(P_2)\right), e_i\>\\
    &= \< (\ev_\eta\circ \tilde{\partial}_\eta)(P_1)\cdot \ev_\x(P)\ev_\x(P_2) +\ev_\x(P_1)\cdot (\ev_\eta\circ\tilde{\partial}_\eta)(P)\cdot \ev_\x(P_2)+\ev_\x(P_1)\ev_\x(P)\cdot (\ev_\eta \tilde{\partial}_\eta)(P_2), e_i\>\\
    &= \<\ev_\x(P_1)\cdot (\ev_\eta\circ\tilde{\partial}_\eta)(P)\cdot  \ev_\x(P_2) , e_i\>\\
    &= \<(\ev_\eta\circ\tilde{\partial}_\eta)(P), \ev_\x(P_1)^\ast e_i \ev_\x(P_2)^\ast \> = \<(\ev_\eta\circ\tilde{\partial}_\eta)(P), \ev_\eta((P_1^*\otimes P_2^*)_i) \>
\end{align*}
Since the image of $\ev_\eta$ is dense in $M\boxtimes_\eta M$, hence dense in $L^2(M\boxtimes_\eta M,\tau)$, the above computation shows that $\ev_\eta\circ\tilde{\partial}_\eta(P)=0$.
Thus, $\tilde{\partial}_\eta$ factors through to the quotient $B\<\mathbf{t}\> / \ker(\ev_\x)\cong B\<\x\>$ and the resulting (well-defined) map is precisely $\pe$. $\hfill\blacksquare$
\end{rem}

%%%%%%%%%%%%%%%%%%%%%%%%%%%%%%%%%%%%%%%%%%%%%%%%%%%
%%%%%        CUMULANT CHARACTERIZATION       %%%%%%
%%%%%%%%%%%%%%%%%%%%%%%%%%%%%%%%%%%%%%%%%%%%%%%%%%%
We will next produce a characterization for the $(B,\eta)$-conjugate system using $B$-valued cumulants, which will help us yield concrete examples. First note that the $B$-valued cumulant function $\kappa_B$ can be partially extended to $L^2(M,\tau)$  as follows. For $a\in M$ we have  
    \[
        \| E_B(a)\| _\tau^2 =\tau\left(E_B(a)^\ast E_B(a)\right)\leq \tau\circ E_B(a^\ast a) =\tau(a^\ast a)=\| a\| _\tau^2
    \]
so that $E_B$ can be extended to a bounded linear map $E_B\colon L^2(M,\tau)\to L^2(B,\tau)$. Moreover, this map remains $B$-bimodular: for $b_1,b_2\in B$ and $(a_n)_{n\in \N}\subset M$ converging to $\xi\in L^2(M,\tau)$ one has
    \[
        E_B(b_1\cdot \xi \cdot b_2) = \lim_{n\to\infty} E_B(b_1 a_n b_2) = \lim_{n\to\infty} b_1E_B(a_n)b_2 = b_1\cdot E_B(\xi)\cdot b_2.
    \]
It follows that $E_B^{(d)}$ for each $d\in \N$ can be extended to allow one input to be from $L^2(M,\tau)$ (and the rest from $M$), with the output valued in $L^2(B,\tau)$. Using the moment-cumulant formula, one can then likewise extend each $\kappa_B^{(d)}$, and $\tau$ also has been extended to $L^2(M,\tau)$. We have the following cumulant characterization.

\begin{thm}[Theorem~\ref{thm c}]\label{thm:cumulants}
Let $(M,\tau)$ be a tracial von Neumann algebra generated by a von Neumann algebra $B$ and a tuple of self-adjoint operators $\x=(x_i)_{i\in I}$. Then $\xi_i\in L^2(M,\tau)$ is the $i$-th $(B,\eta)$-conjugate variable for $\x$ if and only if for all $b,b_1,\ldots, b_d\in B$ and $j,i_1,\ldots, i_d\in I$ one has
\begin{align}\label{eq:cumulants char}
    \begin{cases}
        \cum^{(1)}(\xi_i b) = 0,  \\
        \cum^{(2)}(\xi_i \otimes bx_j)=\eta_{ij}(b), \\ 
        \cum^{(d+1)}(\xi_i \otimes b_1x_{i_1} \otimes \cdots \otimes b_dx_{i_d})=0 ,d\geq 2
    \end{cases}.
\end{align}
\end{thm}

\begin{proof}
If $\xi_i$ is a conjugate variable for $x_i$, for any $b'\in B$,$$\tau(\cum^{(1)}(\xi_i b)b')= \tau(E_B(\xi_i b)b')=\< \xi_i , bb'\>_\tau=\< e_i , \pe(bb')\> =0,$$ hence $\cum^{(1)}(\xi_i b)=0$. The second part is, using the equation (\ref{pre:cumulant 2}),
\begin{align*}
    \tau\left(\cum^{(2)}(\xi_i \otimes bx_j)b'\right) &= \tau\left(E_B^{(2)}(\xi_i \otimes bx_jb')-\cum^{(1)}(\xi_i)\cum^{(1)}(bx_jb')\right)\\\
    &=\tau\left(E_B(\xi_ibx_jb')\right)\\
    &=\<\xi_i, bx_jb'\>_\tau\\
    &=\< e_i ,\pe(bx_jb')\>\\
    &= \< e_i ,be_jb'\> = \tau\left(\eta_{ij}(b)b'\right)
\end{align*}
The remaining formulas will be established by induction on $d\geq 2$. 
%When $d=2$, by the equation (\ref{pre:cumulant 3}) one has
% \begin{align*}
%     \cum^{(3)}(\xi_i\otimes b_1x_{i_1}\otimes b_2x_{i_2})&=E_B(\xi_ib_1x_{i_1}b_2x_{i_2})-E_B(\xi_i)E_B(b_1x_{i_1}b_2x_{i_2})-E_B(\xi_iE_B(b_1x_{i_1})b_2x_{i_2})\\&-E_B(\xi_ib_1x_{i_1})E_B(b_2x_{i_2})+2E_B(\xi_i)E_B(b_1x_{i_1})E_B(b_2x_{i_2})\\
%     &=E_B(\xi_ib_1x_{i_1}b_2x_{i_2})-E_B(\xi_iE_B(b_1x_{i_1})b_2x_{i_2})-E_B(\xi_ib_1x_{i_1})E_B(b_2x_{i_2})\\
% \end{align*}
% so that
% \begin{align*}
%     \tau(\cum^{(3)}(\xi_i\otimes b_1x_{i_1}\otimes b_2x_{i_2})b')&=\<\xi_i, b_1x_{i_1}b_2x_{i_2}b'\> - \<\xi_i, E_B(b_1x_{i_1})b_2x_{i_2}b'\> -\<\xi_i,b_1x_{i_1}E_B(b_2x_{i_2}b')\>\\
%     &=\tau(\eta_{ii_1}(b_1)b_2x_{i_2}b')+\tau(\eta_{ii_2}(b_1x_{i_1}b_2)b')\\
%     &\quad -\tau(\eta_{ii_2}(E_B(b_1x_{i_1})b_2)b')-\tau(\eta_{ii_1}(b_1)E_B(b_2x_{i_2}b'))\\
%     &=0
% \end{align*}
For $d>2$, recall from the moment-cumulant formula (\ref{pre:moment cumulant d}), one has  
\begin{align*} 
    E_B^{(d+1)}(\xi_i\otimes b_1x_{i_1}\otimes \cdots \otimes b_dx_{i_d}) =\cum^{(d+1)}(\xi_i\otimes b_1x_{i_1}\otimes \cdots \otimes b_kx_{i_d} ) + \sum_{\substack{\pi\in NC(d+1)\\ \pi \not\equiv 1_{d+1}}}\cum (\pi)[\xi_i\otimes b_1x_{i_1}\otimes \cdots \otimes  b_dx_{i_d}].
\end{align*}
Then assume by induction that $\cum^{(k+1)}(\xi_i\otimes b_1x_{i_1}\otimes \cdots\otimes b_{k}x_{i_k})=0$ for all $2\leq k\leq d-1$. Then
$\xi_i$ has to be in a block of $\pi$ of size $2$ in order for $\cum(\pi)[\xi_i\otimes b_1x_{i_1}\otimes \cdots \otimes b_kx_{i_k}]$ to be non-zero. Decompose each such $\pi$ as $\pi=\{1,k+1\}\cup \pi_1\cup \pi_2$, where the block $\{1,k+1\}$ connects $\xi_i$ and $x_k$, $\pi_1\in NC(k-1)$ is viewed as a partition on $\{2,\ldots, k\}$ and $\pi_2\in NC(d-k)$ is viewed as a partition on $\{k+2,\ldots, d+1\}$. Then the second term of above equation becomes
\begin{align}\label{the:cumulants eq1}
    \sum_{k=1}^d\cum^{(2)}\bigg(\xi_i \sum_{\pi_1\in NC(k-1)}\cum(\pi_1)[b_1x_{i_1}\otimes \cdots\otimes  b_{k-1}&x_{i_{k-1}}] \\
   & \otimes b_kx_{i_k} \sum_{\pi_2\in NC(d-k)}\cum(\pi_2)[b_{k+1}x_{i_{k+1}}\otimes \cdots \otimes b_dx_{i_d}] \bigg)\nonumber
\end{align}

\begin{align*}
    &=\sum_{k=1}^d\cum^{(2)} \left( \xi_i E_B^{(k-1)}\left(b_1x_{i_1}\otimes \cdots \otimes b_{k-1}x_{i_{k-1}}\right)\otimes b_kx_{i_k} E_B^{(d-k)}\left(b_{k+1}x_{i_{k+1}}\otimes \cdots \otimes b_dx_{i_d}\right)\right) \\
   & =\sum_{k=1}^d\eta_{ii_k}(b_1x_{i_1}\cdots b_{k-1}x_{i_{k-1}}b_k) E_B(b_{k+1}x_{i_{k+1}} \cdots  b_dx_{i_d}).
\end{align*}
Thus for any $b'\in B$ we have
\begin{align*}
    &\tau\left( E_B^{(d+1)}(\xi_i\otimes b_1x_{i_1} \otimes \cdots \otimes b_dx_{i_d})b'\right) \\
    &=\tau\left(\cum^{(d+1)}(\xi_i \otimes  b_1 \otimes \cdots \otimes b_d x_{i_d})b'\right)+\sum_{k=1}^d\tau\left(\eta_{ii_k}(b_1x_{i_1}\cdots b_{k-1}x_{i_{k-1}}b_k) E_B(b_{k+1}x_{i_{k+1}} \cdots  b_dx_{i_d}) b'\right).
\end{align*}
On the other hand, by the definition of a $(B,\eta)$-conjugate variable we have
    \[\tau\left( E_B^{(d+1)}(\xi_i\otimes b_1x_{i_1} \otimes \cdots \otimes b_dx_{i_d})b'\right)=\sum_{k=1}^d\tau\left(\eta_{ii_k}(b_1x_{i_1}\cdots b_{k-1}x_{i_{k-1}}b_k) E_B(b_{k+1}x_{i_{k+1}} \cdots  b_dx_{i_d}) b'\right)\]
So we have $\tau\left(\cum^{(d+1)}(\xi_i \otimes  b_1 x_{i_1}\otimes \cdots \otimes b_d x_{i_d})b'\right)=0$, therefore $\cum^{(d+1)}(\xi_i \otimes  b_1x_{i_1} \otimes \cdots \otimes b_dx_{i_d})$ is zero for all $d\geq 2$.

Conversely, assume $\xi_i$ satisfies (\ref{eq:cumulants char}). It suffices to check $\<\xi_i,p\>_\tau = \<e_i, \pe(p)\>$ for polynomials of the form $p=b_0x_{i_1}b_1\cdots b_dx_{i_d}$. Via the moment-cumulant formula, we have
\begin{align*}
    \<\xi_i,p\>_\tau&=\tau(E_B(\xi_ib_0x_{i_1}b_1\cdots b_dx_{i_d}))=\tau(E_B^{(d+1)}(\xi_i\otimes b_0x_{i_1}\otimes \cdots \otimes x_{i_d}b_d))\\
    &=\sum_{\pi\in NC(d+1)}\tau\left(\cum (\pi)[\xi_i\otimes b_0x_{i_1}\otimes \cdots \otimes x_{i_d}b_d]\right)
\end{align*}
Since the cumulants of order $d\neq 2$ are all zero, the same calculation as (\ref{the:cumulants eq1}) gives \[\<\xi_i,p\>_\tau=\sum_{k=1}^d\tau\left(\eta_{ii_k}(b_0x_{i_1}\cdots x_{i_{k-1}}b_{k-1}) E_B(b_kx_{i_{k+1}} \cdots  x_{i_d}b_d)\right) = \<e_i,\pe(p)\>.\]
Hence $\xi_i$ is the $i$-th $(B,\eta)$-conjugate variable for $\x$.
\end{proof}

\begin{ex}\label{ex:semicircular}
    Let $\mathbf{s}=(s_i)_{i\in I}$ is a $B$-valued semicircular family with covariance $\eta$. Then $\mathbf{s}$ is a $(B,\eta)$-conjugate system for itself. That is, $\<s_i, p\>_2=\<e_i,\pe(p)\>$ for any $p\in B\<\mathbf{s}\>$. To see this, recall from \cite[Definition 4.2.3]{MR1407898} that the definition of such a family, which he calls $B$-Gaussian with covariance matrix $\eta$, is one satisfying $\cum^{(2)}(s_i\otimes bs_j)=\eta_{ij}(b)$ and $\cum^{(d+1)}(s_i\otimes b_1s_1\otimes \cdots \otimes b_ds_d)=0$ for $d\neq 2$. (This is only stated for a finite set $I$ but it naturally extends to infinite sets since any given cumulant only involves finitely many indices.) Thus Theorem~\ref{thm:cumulants} implies $\mathbf{s}$ is a $(B,\eta)$-conjugate system. This generalizes that a $B$-semicircular operator satisfies an integration-by-parts formula from \cite[Proposition 3.10]{MR1737257}.

\end{ex}

\begin{ex}\label{ex:amalgamation}
    Let $\mathbf{s}\subset \Phi(B,\eta)$ be a $B$-valued semicircular operators with covariance $\eta$ (see Section~\ref{subsec:B-valued semicircular}) and consider the amalgamated free product
    \[N=M\ast_B \Phi(B,\eta).\]
    For $t>0$ and each $i\in I$ define $x_i(t):=x_i+\sqrt{t}s_i$, and let $N_t=B\<x_i(t)\colon i\in I\>''$. The trace on $M$ extends to $N$, and so there exists a unique trace preserving conditional expectation $E_t:N\to N_t$. Also, $E_B=E_B\circ E_t$ follows from $E_B$ and $E_t$ both being trace preserving. Then $(\frac{1}{\sqrt{t}}E_t(s_i))_{i\in I}$ is a $(B,\eta)$-conjugate system for $(x_i(t))_{i\in I}$. In fact, by the Theorem~\ref{thm:cumulants}, it suffices to compute its $B$-valued cumulants. Note that $E_B(\frac{1}{\sqrt{t}}E_t(s_i)a)=E_B\circ E_t(\frac{1}{\sqrt{t}}s_ia)=E_B(\frac{1}{\sqrt{t}}s_i a)$ for $a\in N_t$. Also, it was shown in Example~\ref{ex:semicircular} that $s_i$ is the $i$-th $(B,\eta$)-conjugate variable for $\mathbf{s}$, so its $B$-valued cumulants are given by Theorem~\ref{thm:cumulants}. Using this and freeness with amalgamation over $B$,  we have for $b\in B$
    \[\cum^{(1)}\left(\frac{1}{\sqrt{t}}E_t(s_i)b\right)=E_B\left(\frac{1}{\sqrt{t}}E_t(s_i)b\right)= \frac{1}{\sqrt{t}} E_B(s_i)b=0,\]
and 
\begin{align*}
    \cum^{(2)}\left(\frac{1}{\sqrt{t}}E_t(s_i)\otimes b_1x_{i_1}(t)\right) &=E_B\left(\frac{1}{\sqrt{t}}s_ib_1x_{i_1}(t)\right)-E_B\left(\frac{1}{\sqrt{t}}s_i\right)E_B\left(b_1x_{i_1}(t)\right)\\
    &=\cum^{(2)}\left(\frac{1}{\sqrt{t}}s_i\otimes b_1x_{i_1}(t)\right)\\
    &=\cum^{(2)}\left(\frac{1}{\sqrt{t}}s_i\otimes b_1(\sqrt{t}s_{i_1})\right) \\
    &=\eta_{ii_1}(b_1).
\end{align*}
    Similarly,  
\begin{align*}
    \cum^{(d+1)}\left(\frac{1}{\sqrt{t}}E_t(s_i)\otimes b_1x_{i_1}(t)\otimes \cdots \otimes b_dx_{i_d}(t)\right)
    &=\cum^{(d+1)}\left(\frac{1}{\sqrt{t}}s_i\otimes b_1x_{i_1}(t)\otimes \cdots \otimes b_dx_{i_d}(t)\right)\\
    &=\cum^{(d+1)}\left(\frac{1}{\sqrt{t}}s_i\otimes b_1(\sqrt{t}s_{i_1})\otimes \cdots \otimes b_d(\sqrt{t}s_{i_d})\right)\\
    &=0
\end{align*}
    for $d\geq 2$. Hence we have a $(B,\eta)$-conjugate system $(\frac{1}{\sqrt{t}}E_t(s_i))_{i\in I}$ for $(x_i(t))_{i\in I}$. 
\end{ex} 
Using Theorem~\ref{thm:cumulants}, we characterize a conjugate system on tensor products.
\begin{thm}
    Let $(M,\tau)$ be a tracial von Neumann algebra generated by $B$ and a tuple of self-adjoint operators $\x=(x_i)_{i\in I}$ and $(N,\tilde{\tau})$ be another tracial von Neumann algebra. Then $\x$ admits a $(B,\eta)$-conjugate system if and only if $\x\otimes 1 $ admits a $(B\overline\otimes N,\eta\otimes 1_N)$-conjugate system.
\end{thm}
\begin{proof}
    Assume that $\xi_i$ is a $i$-th $(B,\eta)$-conjugate variable for $\x$. By Theorem~\ref{thm:cumulants}, it suffices to check cumulants. Note that for $y_i\in L^2(M,\tau), z_i\in N$, $i=1,...,d$,
    \begin{align*}
        \kappa_{B\otimes N}^{(d)}\left( (y_1\otimes z_1)\otimes \cdots \otimes (y_d\otimes z_d) \right) &= \kappa_{B\otimes N}^{(d)}\left( \left((y_1\otimes 1)\otimes (1\otimes z_1) \right)\otimes \cdots \otimes \left( (y_d\otimes 1)\otimes (1\otimes z_d) \right) \right) \\
        &=  \kappa_{B\otimes N}^{(d)}\left( (y_1\otimes 1)\otimes   \cdots \otimes (y_d\otimes 1)\cdot (1\otimes z_1\cdots z_d)  \right) \\
        &= \kappa_{B\otimes N}^{(d)}\left( (y_1\otimes 1)\otimes  \cdots \otimes (y_d\otimes 1)\right)\cdot (1\otimes z_1\cdots z_d) \\
        &= (\cum^{(d)}(y_1\otimes \cdots \otimes y_d)\otimes 1)\cdot (1\otimes z_1\cdots z_d)\\
        &= \cum^{(d)}(y_1\otimes \cdots \otimes y_d)\otimes z_1\cdots z_d.
    \end{align*}
    Then $\xi_i\otimes 1$ is $i$-th $(B\overline\otimes N, \eta\otimes 1_N)$-conjugate variable for $\x \otimes 1$ since 
    \[\kappa_{B\otimes N}^{(d+1)}\left( (\xi_i\otimes 1)\otimes (b_1x_{i_1}\otimes z_1)\otimes \cdots \otimes  (b_dx_{i_d}\otimes z_d) \right) = \cum^{(d)}(\xi_i\otimes  b_1x_{i_1} \otimes \cdots \otimes b_dx_{i_d})\otimes z_1\cdots z_d=0\] for $d\neq 1$ and
    \[\kappa^{(2)}_{B\otimes N}((\xi_i\otimes 1)\otimes (bx_j\otimes y))=\kappa^{(2)}_{B\otimes N}((\xi_i\otimes 1)\otimes((bx_j\otimes 1)\otimes (1\otimes y)))=\cum^{(2)}(\xi_i\otimes bx_j)\otimes y = \eta_{ij}(b)\otimes y. \]
   
    Conversely, assume that $\zeta_i$ is an $i$-th $(B\overline\otimes N, \eta\otimes 1_N)$-conjugate system for $\x\otimes 1$ and let $q$ be the projection from $L^2(M\overline\otimes N)$ onto $L^2(M)\cong L^2(M)\otimes 1_N$. Then $q\zeta_i$ is an $i$-th $(B,\eta)$-conjugate variable for $\x$. In fact, for a monomial $p=b_0x_{i_1}b_1\cdots x_{i_d}b_d \in \bx$, one has 
    \begin{align*}
        \< q\zeta_i, b_0x_{i_1}b_1\cdots x_{i_d}b_d\>_\tau&=\<\zeta_i,b_0x_{i_1}\cdots x_{i_d}b_d\otimes 1\>_{\tau\otimes \tilde{\tau}}\\
        &= \sum_{k=1}^d\tau\otimes \tilde{\tau}\left(\eta_{ii_k}(b_0x_{i_1}\cdots b_{k-1})b_k\cdots x_{i_d}b_d \otimes 1\right)\\
        &= \sum_{k=1}^d\tau \left(\eta_{ii_k}(b_0x_{i_1}\cdots b_{k-1})b_k\cdots x_{i_d}b_d \right). 
    \end{align*}
    This completes the proof.
\end{proof}

\begin{cor}\label{cor: scalar}
     Let $(M,\tau)$ be a tracial von Neumann algebra generated by $B$ and a tuple of self-adjoint operators $\x=(x_i)_{i\in I}$. Then $\x$ admits a $(\C, (\delta_{ij}\tau)_{i,j\in I})$-conjugate system if and only if $\x\otimes 1$ admits a $(\C\otimes B, (\tau\otimes 1_B)_{ij})$-conjugate system.
\end{cor}

\begin{cor}
    Let $(M,\tau)$ be a tracial von Neumann algebra generated by $B$ and a tuple of self-adjoint operators $\x=(x_i)_{i\in I}$ with $|I|>1$. Then $\x$ admits a $(\C, (\delta_{ij}\tau)_{i,j\in I})$-conjugate system if and only if $\C\<\x\>$ and $B$ are tensor independent with respect to $\tau$ and $\x$ admits a $(B,(\delta_{ij}E_B)_{i,j})$-conjugate system.
\end{cor}
\begin{proof}
    Assume $\x$ admits a scalar conjugate system and let $A=W^\ast(\x)$ be a von Neumann algebra generated by $x_i$'s. For any positive element $b\in B$, define a trace $\tau_b$ on $A$ by $\tau_b(a)=\tau(ab)$. However, by \cite[Theorem 1]{MR2606868}, $A$ is a factor and hence it has a unique trace. Thus, there exists $c\in \C$ such that $\tau_b(a)=c\tau(a)$ for all $a\in A$. Since $c=c\tau(1)=c\tau_b(1)=\tau(b)$, we have 
    $\tau(ab)=\tau_b(a)=\tau(b)\tau(a)$. This implies $A$ and $B$ are independent with respect to $\tau$. Next we claim $A\overline{\otimes}B \cong A\vee B$. Let $\pi$ be a mapping from the algebraic tensor product of $A$ and $B$, $A\odot B$, into $M=A\vee B$ that sends $a\otimes b$ to $ab$. Then $\pi$ is well defined since 
    assuming $\sum_ja_j\otimes b_j=\sum_k c_k\otimes d_k$ one has
    \begin{align*}
        \< \sum_j a_jb_j - \sum_k c_kd_k, xy\>_\tau &= \sum_j\tau(x a_j^\ast b_j^\ast y)-\sum_k \tau(x c_k^\ast d_k^\ast y) \\
        &= \sum_j\tau(x a_j^\ast)\tau(b_j^\ast y) -\sum_k\tau(x c_k^\ast)\tau(d^\ast_k y)\\
        &=\sum_j\<a_i\otimes b_j,x\otimes y\>_{\tau\otimes \tau} - \sum_k\<c_k\otimes d_k, x\otimes y\>_{\tau\otimes \tau}=0
    \end{align*}
    for $x\in A$ and $y\in B$. Then $\pi$ is a $\ast$-homomorphism and extends to a normal isomorphism since $\tau\circ \pi= \tau\otimes \tau$ by independence. Thus $A\overline{\otimes}B \cong A\vee B$. Then Corollary~\ref{cor: scalar} gives us that $\x=\x\otimes 1$ admits a $(\C\otimes B, (\delta_{ij}\tau \otimes 1)_{ij})$-conjugate system. Note that identifying $A\otimes B$ and $A\vee B$ carries the covariance matrix $\tau\otimes 1$ to $E_B$. Therefore, $\x$ admits a $(B, (\delta_{ij}E_B)_{ij})$-system.
    The opposite direction follows from Corollary~\ref{cor: scalar} after noting that $E_B=\tau\otimes 1 $ under the identification $M\cong A\overline{\otimes} B$.
\end{proof}
%%%%%%%%%%%%%%%%%%%%%%%%%%%%%%%%%%%%%%
%%%%%          SECTION 3       %%%%%%
%%%%%%%%%%%%%%%%%%%%%%%%%%%%%%%%%%%%%%
\section{Absence of atoms in the presence of a conjugate system}\label{section: no atoms}
Recall that for a self-adjoint element $x$, a self-adjoint $b\in B$ is called an \emph{atom} of $x$ in $B$ if $\ker(x-b)\neq 0$.
In this section we argue that certain self-adjoint polynomials in $\bx$ do not have atoms in $B$ when $\mathbf{x}$ admits a $(B,\eta)$-conjugate system. This part is crucial to show $\bx$ is diffuse relative to $B$ in Section~\ref{section: relative diffuse}. 
We implement Mai, Speicher and Weber's methodology from \cite{MR3558227}, which showed the absence of zero divisors under the assumption of finite free Fisher information in scalar case. We view $\partial_\eta$ as a densely defined operator on $L^2(M,\tau)$ with codomain $L^2(M\boxtimes_\eta M,\tau)$.
Recall that we have a conjugate linear isometry $J$ on $L^2(M\boxtimes_\eta M,\tau)$ defined by $J(ae_ib)=b^\ast e_i a^\ast$ and which satisfies $J(\pe(p))=\pe(p^\ast)$.

\begin{prop} \label{prop:adjoint} 
Let $(M,\tau)$ be a tracial von Neumann algebra generated by a von Neumann algebra $B$ and a tuple of self-adjoint operators $\x=(x_i)_{i\in I}$. Assume that a $(B,\eta)$-conjugate system exists for $\x$. Then for $p,q \in \bx$ and  $\xi \in \operatorname{span}\{\bx e_i \bx :i\in I\}$ one has $p\xi q\in \dom(\pea)$ with
    \[\pea(p\xi q)=p \pea(\xi)q-\< J\pe(p)\mid \xi\>_M q -p\< J(\xi) \mid\pe(q)\>_M.\]
Consequently, $\partial_\eta$ is closable.
\end{prop}

\begin{proof}
For $p,q\in \bx$ and $\xi \in \operatorname{span}\{\bx e_i \bx :i\in I\}$ and any $r\in \bx$, 
\begin{align*}
    \<\pea(p\xi q),r\>_\tau =\< p\xi q,\pe(r)\> &=   \<  \xi , p^\ast\pe(r)q^\ast\>\\
    &=  \< \xi ,\pe(p^\ast r q^\ast)-\pe(p^\ast)rq^\ast -p^\ast r \pe(q^\ast)\>.\\
\end{align*}
The first term is $\< \xi , \pe(p^\ast rq^\ast)\> = \< \pea(\xi ),p^\ast r q^\ast\>_\tau = \< p\pea(\xi )q,r\>_\tau$. The second term is 
\begin{align*}
    \< \xi , \pe(p^\ast)rq^\ast\> &= \tau(\< \xi  \mid \pe(p^\ast)rq^\ast\>_M) \\ &= \tau( \< \xi  \mid \pe(p^\ast)\>_Mrq^\ast) \\& = \< \< \pe(p^\ast)\mid \xi \>_M , rq^\ast\>_\tau \\&= \< \< \pe(p^\ast) \mid \xi \>_Mq , r\>_\tau\\
    &=\<\<J(\pe(p))\mid \xi\>_Mq,r\>_\tau
\end{align*}
The third term is 
\begin{align*}
    \< \xi , p^\ast r\pe(q^\ast)\> &= \<r^\ast p \xi ,\pe(q^\ast)\>\\
    &=\< J(\pe(q^\ast)),J(r^\ast p \xi )\>\\
    &=\< \pe(q),J(\xi) p^\ast r\>\\
    &=\< \<J(\xi)\mid \pe(q)\>_M,p^\ast r\>_\tau\\
    &= \< p\< J(\xi) \mid\pe(q)\>_M, r\>_\tau.
\end{align*}
Since $r\in \bx$ was arbitrary, this establishes the claimed equality. Since $\pea$ is densely defined, $\pe$ is closable. 
\end{proof}

%%%%%%%%%%%%%%%%%%%%%%%%%%%%%%%%%%%%%%
%%%%%                           %%%%%%
%%%%%%%%%%%%%%%%%%%%%%%%%%%%%%%%%%%%%%
Even though $\pe$ is an unbounded operator on $L^2(M,\tau)$ to $L^2(M\boxtimes_\eta M,\tau)$, we will see that $\| \pea(pe_i)\| _\tau$ can be controlled in by the operator norm of $p$ (see Proposition~\ref{prop:adjoint ineq} below). To see this, we first need to see the following lemma. 

\begin{lem}\label{prop:adjoint2}
Let $(M,\tau)$ be a tracial von Neumann algebra generated by a von Neumann algebra $B$ and a tuple of self-adjoint operators $\x=(x_i)_{i\in I}$. Assume that a $(B,\eta)$-conjugate system exists for $\x$. Then for $p,q\in \bx$, 
    \[\<  \pea(p e_i), \pea(q e_j)\>_\tau = \<  \pea(e_i), \pea(p^\ast q e_j)\>_\tau \]
and 
    \[\< \pea(e_ip), \pea(e_jq) \>_\tau =\<\pea(e_i),\pea(e_jqp^\ast)\>_\tau. \]
\end{lem}

\begin{proof}
Since $\pea$ is $B$-linear, it suffices to show that $\< \pea(pe_i),\pea(qe_j)\>_\tau=\<\pea(p_0e_i),\pea(x_kb^\ast qe_j)\>_\tau$ for $p=bx_kp_0$ with $b\in B$ and $p_0\in \bx$, using Proposition~\ref{prop:adjoint},
\begin{align*}
    \< \pea(bx_kp_0e_i), \pea(qe_j)\>_\tau&=\<bx_k\pea(p_0e_i),\pea(qe_j)\>_\tau -\< \<J\pe(bx_k)\mid p_0e_i\>_M, \pea(qe_j)\>_\tau\\
    &=\<\pea(p_0e_i), x_kb^\ast\pea(qe_j) \>_\tau -\<\<e_kb\mid p_0e_i\>_M, \pea(qe_j)\>_\tau
\end{align*}
Note that the second term $\<\<e_kb\mid p_0e_i\>_M, \pea(qe_j)\>_\tau$ is zero since $\<e_kb\mid p_0e_i\>_M=b\eta_{ki}(p_0)\in B$ and the inner product of any element of B against $\pea(qe_j)$ is zero. Similarly,  
    \[\<\pea(p_0e_i),\pea(x_kb^\ast qe_j)\>_\tau=\<\pea(p_0e_i), x_kb^\ast\pea(qe_j)-\<be_k\mid qe_j\>_M\>_\tau =\<\pea(p_0e_i), x_kb^\ast\pea(qe_j)\>_\tau\]
Thus $\< \pea(bx_kp_0e_i),\pea(qe_j)\>_\tau=\<\pea(p_0e_i),\pea(x_kb^\ast qe_j)\>_\tau$, and the similar argument yields $\< \pea(e_ip), \pea(e_jq) \>_\tau =\<\pea(e_i),\pea(e_jqp^\ast)\>_\tau$.

\end{proof}

%%%%%%%%%%%%%%%%%%%%%%%%%%%%%%%%%%%%%%
%%%%%                           %%%%%%
%%%%%%%%%%%%%%%%%%%%%%%%%%%%%%%%%%%%%%
Using the above lemma, we will now investigate $\|\pea(pe_i)\|_\tau$. The proof of Proposition~\ref{prop:adjoint ineq} is inspired by \cite[Theorem 8.10]{MR3585560}
\begin{prop}\label{prop:adjoint ineq} Let $(M,\tau)$ be a tracial von Neumann algebra generated by a von Neumann algebra $B$ and a tuple of self-adjoint operators $\x=(x_i)_{i\in I}$. Assume that a $(B,\eta)$-conjugate system exists for $\x$. For any $p\in \bx$ and $i\in I$ we have the following :
        \[\| \pea(pe_i)\|_\tau= \|  p\pea(e_i) -\< \pe(p^\ast)\mid e_i\>_M \| _\tau \leq \| \pea(e_i)\| _\tau \| p\| \] 
    and 
        \[\| \pea(e_ip)\| _\tau= \| \pea(e_i)p -\< e_i\mid \pe(p)\>_M \| _\tau \leq \| \pea(e_i)\| _\tau \| p\| .
        \]
    Therefore, 
        \[\| \< \pe(p^\ast)\mid e_i\>_M\| _\tau\leq 2\| \pea(e_i)\| _\tau\| p\|  \qquad \text{ and } \qquad \| \< e_i\mid \pe(p)\>_M\| _\tau \leq 2\| \pea(e_i)\| _\tau\| p\| .\] 
\end{prop}
\begin{proof}
For $p\in \bx$, by Proposition~\ref{prop:adjoint} and the Lemma~\ref{prop:adjoint2} we have
    \begin{align*}
            \|  p\pea(e_i) -\< \pe(p^\ast)\mid e_i\>_M \| _\tau^2 &= \< \pea(pe_i) ,\pea(pe_i) \>_\tau = \< \pea(e_i), \pea(p^\ast p e_i)\>_\tau \leq \|  \pea(e_i)\| _\tau  \| \pea( p^\ast pe_i) \| _\tau.
    \end{align*}
    By iteration on $\| \pea( p^\ast pe_i) \| _\tau$,
    \begin{align}\label{ineq: prop 3.3}
            \|  p\pea(e_i) -\< \pe(p^\ast)\mid e_i\>_M \| _\tau^2 &\leq \|  \pea(e_i)\| _\tau  \|  \pea(e_i)\| _\tau ^{\frac{1}{2}} \|   \|  \pea((p^\ast p)^2e_i) \| _\tau^{\frac{1}{2}}\nonumber\\
            &\vdots\nonumber\\
            &\leq \|  \pea(e_i)\| _\tau ^{1+\frac{1}{2}+\cdots +\frac{1}{2^n}} \| \pea((p^\ast p)^{2^n}e_i)\| _\tau^{\frac{1}{2^n}}.
    \end{align}
    As $n\to \infty$, $\|  \pea(e_i)\| _\tau ^{1+\frac{1}{2}+\cdots +\frac{1}{2^n}}$ converges to $\| \pea(e_i)\| _\tau^2$, and it remains to bound $\| \pea((p^\ast p)^{2^n}e_i)\| _\tau^{\frac{1}{2^n}}$. From   Proposition~\ref{prop:adjoint} we have
        \[
            \|\pea((p^\ast p)^{2^n}e_i)\|_\tau^{\frac{1}{2^n}} = \left\| (p^\ast p)^{2^n}\pea(e_i) -\< \pe((p^\ast p)^{2^n})\mid e_i\>_M \right\|_\tau^{\frac{1}{2^n}}.
        \]
    Toward bounding this, we claim that $\| \< \pe((p^\ast p)^{2^n})\mid e_i\>_M \|_\tau\leq \|\pe((p^\ast p)^{2^n})\|_\tau \|\eta_{ii}(1)\|$. Setting $a=\< \pe((p^\ast p)^{2^n})\mid e_i\>_M $ and using $0\leq \eta_{ii}(1) \leq \|\eta_{ii}(1)\|$ gives us
    \begin{align*}
            \| a\| _\tau^2 &=\|  \< \pe((p^\ast p)^{2^n})\mid e_i\>_M \| _\tau^2=\tau\left(a \< \pe((p^\ast p)^{2^n})\mid e_i\>_M \right)=  \< \pe((p^\ast p)^{2^n}), e_ia^\ast\>\\
            &\leq \| \pe((p^\ast p)^{2^n})\| _\tau\| e_ia^\ast\| _\tau = \| \pe((p^\ast p)^{2^n})\| _\tau \tau(a^\ast \eta_{ii}(1)a)^{\frac12}\leq\| \pe((p^\ast p)^{2^n})\| _\tau  \| a\| _\tau\| \eta_{ii}(1)\| .
    \end{align*}
    Also, note that for $q\in \bx$, $\| \pe(q^k)\| _2\leq k\| q\| ^{k-1}\| \pe(q)\| _2$ by iterating the Leibniz rule.
    Thus,  
    \begin{align*}
            \| \pea((p^\ast p)^{2^n}e_i)\| _\tau^{\frac{1}{2^n}} &= \|  (p^\ast p)^{2^n}\pea(e_i) -\< \pe((p^\ast p)^{2^n})\mid e_i\>_M \| _\tau^{\frac{1}{2^n}} \\
            &\leq \left( \| p^\ast p\| ^{2^n}\| \pea(e_i)\| _\tau +\|  \< \pe((p^\ast p)^{2^n})\mid e_i\>_M \| _\tau  \right)^{\frac{1}{2^n}}\\
            &\leq \left( \| p^\ast p\| ^{2^n}\| \pea(e_i)\| _\tau +\|  \pe((p^\ast p)^{2^n})\| _2 \| \eta_{ii}(1)\|    \right)^{\frac{1}{2^n}}\\
            &\leq \left( \| p^\ast p\| ^{2^n} \| \pea(e_i)\| _\tau +2^n\|  p^\ast p\| ^{2^n-1}\| \pe(p^\ast p)\| _2  \| \eta_{ii}(1)\|  \right)^{\frac{1}{2^n}}\\
            &= \| p^\ast p\|  \left( \| \pea(e_i)\| _\tau +2^n \frac{\| \pe(p^\ast p)\| _2}{\| p^\ast p\| } \| \eta_{ii}(1)\| \right)^{\frac{1}{2^n}}\\
            &= \| p\| ^2 \left( \| \pea(e_i)\| _\tau +2^n \frac{\| \pe(p^\ast p)\| _2}{\| p^\ast p\| } \| \eta_{ii}(1)\| \right)^{\frac{1}{2^n}}.
    \end{align*}
    Sending $n\to \infty$, from the inequality (\ref{ineq: prop 3.3}) we have $\|  p\pea(e_i) -\< \pe(p^\ast)\mid e_i\>_M \| _\tau \leq \| \pea(e_i)\| _\tau \| p\| $.
    A similar argument yields $\| \pea(e_i)p -\< e_i\mid \pe(p)\>_M \| _\tau \leq \| \pea(e_i)\| _\tau \| p\| .$
\end{proof}

%%%%%%%%%%%%%%%%%%%%%%%%%%%%%%%%%%%%%%
%%%%%                           %%%%%%
%%%%%%%%%%%%%%%%%%%%%%%%%%%%%%%%%%%%%%

Recall from Proposition~\ref{prop:adjoint} that $\pe$ is closable when a $(B,\eta)$-conjugate system exists, and we will denote its closure by $\cpe$. From \cite{MR1171180}, $M\cap\dom(\cpe)$ is a $\ast$-algebra and $\cpe$ still satisfies Leibniz rule on $M\cap \dom(\cpe)$ (see also \cite[Section 1.7]{MR2550142}).

\begin{prop}\label{prop:zero}
    Let $(M,\tau)$ be a tracial von Neumann algebra generated by a von Neumann algebra $B$ and a tuple of self-adjoint operators $\x=(x_i)_{i\in I}$. Assume that a $(B,\eta)$-conjugate system exists for $\x$.
    For $P\in M\cap\dom(\cpe)$ and $u,v\in M_{s.a.}$, if $Pu=0$ and $P^\ast v=0$ then $$v\cpe(P)u=0.$$
\end{prop} 
\begin{proof}
    By Kaplansky's density theorem, there exists self-adjoints $u_k,v_k\in \bx$ such that $$\sup_{k\in \mathbb{N}} \| u_k\| , \sup_{k\in \mathbb{N}}\| v_k\| <\infty \mbox{ and }  \lim_{k\to\infty}\| u_k-u\| _\tau
    =0, \lim_{k\to\infty}\| v_k-v\| _\tau=0. $$ For arbitrary $Q_1,Q_2 \in \bx$, and arbitrary $i\in I$,
    \begin{align}\label{eq:zero1}
        \< Pu_k, \pea(v_kQ_1e_iQ_2)\>_\tau = & \< \cpe(Pu_k),v_kQ_1e_iQ_2\> \nonumber \\
        =& \< \cpe(P)u_k +P\pe(u_k),v_jQ_1e_iQ_2 \>\nonumber\\
        =& \< \cpe(P)(u_k-u),v_kQ_1e_iQ_2 \> +\< \cpe(P)u, v_kQ_1e_iQ_2 \> +\< P\pe(u_k), v_kQ_1e_iQ_2\>\nonumber \\
       =&\< \cpe(P)(u_k-u),v_kQ_1e_iQ_2\> +\< \cpe(P)u,(v_k-v)Q_1e_iQ_2\>\nonumber \\ &+\< \cpe(P)u, vQ_1e_iQ_2\> +\< P\pe(u_k), v_kQ_1e_iQ_2\> \nonumber\\         
        =&\< \cpe(P)(u_k-u),v_kQ_1e_iQ_2\> +\< \cpe(P)u,(v_k-v)Q_1e_iQ_2\> \nonumber\\ &+\< v\cpe(P)u, Q_1e_iQ_2\> +\< \pe(u_k)Q_2^\ast,P^\ast v_kQ_1e_i\>
    \end{align}
    We are going to show the third term $\< v\cpe(P)u, Q_1e_iQ_2\>$ is zero by showing the remaining three terms and the original expression all converge to 0 as $k\to \infty$.
    To estimate the term on the left hand side, note that from Proposition~\ref{prop:adjoint}, for $y_1,y_2\in \bx$
        \[\pea(y_1e_iy_2)=\pea(y_1e_i)y_2-y_1\< e_i  \mid \pe(y_2)\>_M\] 
    and by Proposition~\ref{prop:adjoint ineq}, 
        \[\| \pea(y_1e_iy_2)\| _\tau \leq 3\| \pea(e_i)\| _\tau\cdot\| y_1\| \cdot \| y_2\| .\]
    Hence 
    \begin{align*}
    |\< Pu_k,\pea(v_kQ_1e_iQ_2)\>_\tau | &\leq \| Pu_k\| _2 \cdot \| \pea(v_kQ_1e_iQ_2)\| _2 \\ 
    &\leq \| Pu_k-Pu\| _\tau\cdot 3\| \pea(e_i)\| _\tau\cdot \| v_k\| \cdot \| Q_1\| \cdot \| Q_2\| \\ 
    &\leq \| P\| \cdot \| u_k-u\| _\tau \cdot 3\| \pea(e_i)\| _\tau\cdot \| v_k\| \cdot \| Q_1\| \cdot \| Q_2\| ,
    \end{align*}
    which tends to zero as $k\to \infty$. For the term $\< \pe(u_k)Q_2^\ast,P^\ast v_kQ_1e_i\>$ on the right hand side, consider for $y_1,y_2\in\bx$ and $y_3\in M\cap\dom(\cpe)$,
       \[ \< \pe(y_1)y_2, y_3e_i\> = \< \pe(y_1y_2),y_3e_i\> - \< y_1 \pe(y_2), y_3e_i\>. \]
    Using the conjugate linear isometry $J(ae_ib)=b^\ast e_i a^\ast$,
    
    \begin{multicols}{2}
    \noindent
    \begin{equation*}
    \begin{split}
        \< \pe(y_1y_2),y_3e_i\> &= \< J(y_3e_i), J(\pe(y_1y_2))\>\\
        &= \< e_iy_3^\ast ,\pe(y_2^\ast y_1^\ast)\>\\
        &=\tau(y_3\<e_i\mid \pe(y_2^\ast y_1^\ast)\>\\
        &=\< \< \pe(y_2^\ast y_1^\ast)|e_i\>_M, y_3 \>_\tau    
    \end{split}
    \end{equation*}
    \columnbreak
    \begin{equation*}
    \begin{split}
        \< y_1 \pe(y_2), y_3e_i\> &= \< \pe(y_2),  y_1^\ast y_3 e_i \>\\
        &=\< J(y_1^\ast y_3 e_i), J(\pe(y_2))\>\\
        &=\<e_iy_3^\ast y_1, \pe(y_2^\ast)\>\\
        &=\< \< \pe(y_2^\ast)\mid e_i\>_M, y_1^\ast y_3\>_\tau
    \end{split}
    \end{equation*}
    \end{multicols}
    
    \noindent Thus we have 
    \[\< \pe(y_1)y_2, y_3e_i\> = \< \< \pe(y_2^\ast y_1^\ast)|e_i\>_M, y_3 \>_\tau - \< \< \pe(y_2^\ast)|e_i\>_M, y_1^\ast y_3\>_\tau\] 
    and by Proposition~\ref{prop:adjoint ineq},
       \begin{align*}
                |\< \pe(y_1)y_2, y_3e_i\>| &\leq \| \< \pe(y_2^\ast y_1^\ast)\mid e_i\>_M\| _\tau \cdot \| y_3\| _\tau + \| \< \pe(y_2^\ast)\mid e_i\>_M\| _\tau\cdot\| y_1^\ast y_3\| _\tau\\
                &\leq 2\| \pea(e_i)\| _\tau\cdot \| y_2^\ast y_1^\ast\| \cdot \| y_3\| _\tau + 2\| \pea(e_i)\| _\tau\cdot\| y_2^\ast\| \cdot \| y_1^\ast y_3\| _\tau\\
                &\leq  4\| \pea(e_i)\| _\tau\cdot \| y_3\| _\tau\cdot \| y_1\| \cdot \| y_2\| .
        \end{align*}
    Using this inequality,  we can bound the term $\< \pe(u_k)Q_2^\ast,P^\ast v_kQ_1e_i\>$ as:
       \begin{align*}
                |\< \pe(u_k)Q_2^\ast,P^\ast v_kQ_1e_i\>| &\leq 4\| \pea(e_i)\| _\tau\cdot \| P^\ast v_k Q_1\| _\tau\cdot\| u_k\| \cdot \| Q_2\| \\
                &=4\| \pea(e_i)\| _\tau\cdot \| P^\ast (v_k-v) Q_1\| _\tau\cdot\| u_k\| \cdot \| Q_2\| \\
                &\leq 4\| \pea(e_i)\| _\tau\cdot \| P^\ast\| \cdot  \| v_k -v\| _\tau \cdot \| Q_1\| \cdot\| u_k\| \cdot \| Q_2\| 
        \end{align*}
    which converges to $0$ as $k\to\infty$.
    The remaining terms $\< \cpe(P)u,(v_k-v)Q_1e_iQ_2\>$ and $\< \cpe(P)(u_k-u),v_kQ_1e_iQ_2\>$ also converge to 0 by the Cauchy-Schwarz inequality and normality of the bimodule actions. Summing up, every term converges to 0 as $k\to \infty$, so we conclude that the third term in the last expression of equation (\ref{eq:zero1}), $\< v\cpe(P)u, Q_1e_iQ_2\>$, is zero. Since $Q_1$ and $Q_2$ were arbitrary, we have $v\cpe(P)u=0$.
\end{proof}

%%%%%%%%%%%%%%%%%%%%%%%%%%%%%%%%%%%%%%
%%%%%                           %%%%%%
%%%%%%%%%%%%%%%%%%%%%%%%%%%%%%%%%%%%%%

\begin{thm}[Theorem~\ref{thm b}]\label{thm:no atom}
Let $(M,\tau)$ be a tracial von Neumann algebra generated by a von Neumann algebra $B$ and a tuple of self-adjoint operators $\x=(x_i)_{i\in I}$. Assume that a $(B,\eta)$-conjugate system exists for $\x$ and a covariance matrix $\eta$ with $\eta_{ii}=E_B$ for all $i\in I$. Then a self-adjoint $B$-linear combination $P=\sum_{j}a_jx_ib_j +b_0$ has no atoms in $B$ when $a_j,b_j\in B$ and there exists a positive scalar $c$ such that $\sum_{j}a_jb_j\geq c\cdot 1$.

If we furthermore assume $\eta=(\delta_{ij}E_B)_{i,j\in I}$, then a self-adjoint $B$-linear combination $P=\sum_{i}\sum_{j}a_j^{(i)}x_ib_j^{(i)} +b_0$ has no atoms in $B$ when $a_j^{(i)},b_j^{(i)}\in B$ and there exists an $1\leq i\leq n$ and a positive scalar $c$ such that $\sum_{j}a_j^{(i)}b_j^{(i)}\geq c\cdot 1$.
\end{thm}
\begin{proof}
Assume $\eta_{ii}=E_B$ and that the $i$-th $(B,\eta)$-conjugate variable exists. For a self-adjoint $B$-linear combination $P=\sum_{j}a_jx_ib_j +b_0$, let $b$ be a self-adjoint element of $B$, and $w$ be the projection onto $\ker(P-b)$ so that $(P-b)w=0$ and  $(P-b)^*w=(P-b)w=0$. Then $w \pe(P-b) w=0$ by Proposition~\ref{prop:zero} and
    \begin{align*}
            0=\< e_i , w\pe(P)w\> =\<e_i, \sum_j wa_je_ib_j w\>&=\tau\left(E_B(w)\sum_j a_jb_j w \right)\\&=\tau\left( E_B(w)\sum_j a_jb_j E_B(w)\right) \geq c \|E_B(w)\|_2^2.
    \end{align*}
Thus we have $E_B(w)=0$ and therefore $w=0$ by the faithfulness of $E_B$.

Next, assume $\eta_{ij}=\delta_{ij}E_B$ and that a $(B,\eta)$-conjugate system exists. For a self-adjoint $B$-linear combination $P=\sum_{i}\sum_{j}a_j^{(i)}x_ib_j^{(i)} +b_0$, let $b$ be a self-adjoint element of $B$ and $w$ be as before so that $w \pe(P-b) w=0$ by Proposition~\ref{prop:zero}. Then
    \begin{align*}
            0=\< e_i , w\pe(P)w\>=\<e_i, w\sum_i\sum_j a_j^{(i)}e_ib_j^{(i)}w\> &=\tau\left(E_B(w)\sum_j a_j^{(i)}b_j^{(i)} w \right)\\
            &=\tau\left( E_B(w)\sum_j a_j^{(i)}b_j^{(i)} E_B(w)\right) \geq c \|E_B(w)\|_\tau^2.
    \end{align*}
Thus $E_B(w)=0$ and therefore $w=0$.
\end{proof}

\begin{rem}
    Suppose $\eta_{ii} =\sum_k(\zeta^{(i)}_k)_\ast \circ (\zeta^{(i)}_k)$
    where each $\zeta^{(i)}_k$ is a normal completely positive map whose predual map still restricts to $M$ and assume that $\sum_k \| \zeta^{(i)}_k(x)\| _\tau^2 =0$ implies $x=0$. (Note that $E_B=(E_B)_*\circ E_B$ is of this form.)  If the $i$th $(B,\eta)$-conjugate variable exists, then arguing as in the proof of Theorem~\ref{thm:no atom} it follows that $x_i$ has no atoms in $B$: for any self-adjoint $b\in B$ and $w$ the projection onto the kernel of $x_i-b$, one has
        \begin{align*}
            0&=\<e_i, w \partial_\eta(x_i-b) w\>=\<e_i, w e_i w\>=\tau(\eta_{ii}(w)w)\\
            &=\sum_k\tau\left((\zeta^{(i)}_k)_\ast \circ (\zeta^{(i)}_k)(w)w\right)= \sum_k\tau(\zeta^{(i)}_k(w)\zeta^{(i)}_k(w))= \sum_k\| \zeta^{(i)}_k(w)\| ^2_\tau.\end{align*}
    Thus $w=0$. $\hfill\blacksquare$
\end{rem}
%%%%%%%%%%%%%%%%%%%%%%%%%%%%%%%%%%%%%%
%%%%%                           %%%%%%
%%%%%%%%%%%%%%%%%%%%%%%%%%%%%%%%%%%%%%

\begin{rem}
    One might hope to use the strategy from \cite[Theorem 3.1]{MR3558227} to extend the argument in Theorem~\ref{thm:no atom} to higher degree polynomials by iteratively differentiating one of the highest order terms. However, one runs into the following difficulty in the operator-valued case. 
    %In \cite[Lemma 3.10]{MR3558227} it is shown that if $\ker{x}$ is non-zero, $\ker{x^\ast}$ is non-zero for $x\in M$. If $Pw=0$ holds for a polynomial $P\in \bx$ and $w\neq 0$, then $\{0\}\notin \ker(P)$. This implies $\ker(P^\ast)\neq \{0\}$ and $P^\ast p =0$ where $p$ is the projection onto $\ker(P^\ast)$. Thus we have $p\pe(P)w=0$ by the Theorem~\ref{prop:zero}.
    Define a linear mapping $\Delta_{p,i}:M\to M$ by 
    \[\Delta_{p,i}P=\<e_i \mid p\pe(P)\>_M\] 
    for each $i$ and a projection $p\in M$ satisfying $P^*p=0$. If $Pw=0$, we have $(\Delta_{p,i}P)w=\<e_i \mid p\pe(P)\>_M w =\<e_i \mid p\pe(P)w\>_M=0 $ by Proposition~\ref{prop:zero}. For $P=x_i$ we get 
    \[0=(\Delta_{p,i}P)w=\<e_i\mid pe_i\>_M w=\eta_{ii}(p)w,\]
    which we argued in Theorem~\ref{thm:no atom} implied $w=0$ when $\eta_{ii}=E_B$ and $p=w$. If $\eta_{ij}=0$ for $i\neq j$, then for $P=x_ix_j$ one can iterate this for a properly chosen $p$ (see \cite[Lemma 3.14]{MR3558227}) to obtain
        \[
            0 = \Delta_{p,j}(\Delta_{w,i}(x_ix_j))w = \eta_{jj}(p \eta_{ii}(w))w
        \]
    But even for $\eta=(\delta_{ij} E_B)_{i,j\in I}$ it does not appear that the above implies $w=0$ in general. And this issue only worsens for higher order polynomials. $\hfill\blacksquare$
\end{rem}

Consider a map $\Psi:L^2(M\boxtimes_\eta M,\tau)\to L^2(M\ast_B \Phi(B,\eta),\tau)$ that sends $xe_iy$ to $xs_iy$ where $x,y\in M$ and $\mathbf{s}=(s_i)_{i\in I}$ is a $(B,\eta)$-valued semicircular family. Then $\Psi$ is a $M$-bilinear isometry since for $x_1,y_1,x_2,y_2\in M$, one has
    \begin{align*}
        \<x_1s_iy_1,x_2s_jy_2\>_{L^2(M\ast_B \Phi(B,\eta),\tau)} &= \tau\left( E_B(y_1^\ast s_ix_1^\ast x_2s_jy_2)  \right) = \tau\left( E_B(y_1^\ast)E_B(s_ix_1^\ast x_2s_j)E_B(y_2)  \right) \\&= \tau\left(y_1^\ast\eta_{ij}(x_1^\ast x_2)y_2\right)= \<x_1e_iy_1,x_2e_jy_2\> 
    \end{align*}
by freeness with amalgamation over $B$ and $E_B(s_ias_j)=\eta_{ij}(a)$. We recover Mai-Speicher-Weber's work \cite[Theorem 3.1]{MR3558227} by combining Proposition~\ref{prop:zero}, Corollary~\ref{cor: scalar}, the map $\Psi$ and Anderson's self-adjoint linearization trick (see also \cite[Section 3]{BMS17}).

\begin{rem}
    Let $(M,\tau)$ be a tracial von Neumann algebra, $\x=(x_1,...,x_n)\in M$ be a tuple of self-adjoint operators. Assume $\x$ admit a $(\C, (\delta_{ij}\tau)_{i,j})$-conjugate system (i.e., $\Phi^\ast(x_1,...,x_n)<\infty$). Then for any non-constant self-adjoint polynomial $p$, there exists no nonzero self-adjoint element $w\in (\C\<\x\>)''$ such that $p(x)w=0$.
    
    Indeed, assume $p$ is a non-constant self-adjoint polynomial and $w$ is the projection onto $\ker(p(\x))$ such that $p(\x)w=0$.
    Then 
    \[
        L(\x)=\begin{bmatrix} 0 &u\\ v& Q\end{bmatrix}= a_0\otimes 1 +a_1\otimes x_1 +\cdots a_n\otimes x_n 
    \in \mathbb{M}_k(\C)\otimes M
    \]
    for some $k\in \N$ is a self-adjoint linearization of $p(\x)=-u(\x)Q(\x)^{-1}v(\x)$. Let $B=\mathbb{M}_k(\C)$. Since $\x$ admits $(\C,\delta_{ij}\tau)$-conjugate system, $1\otimes \x $ admits $(B\otimes \C, (\delta_{ij}1_B\otimes \tau))$-conjugate system by Corollary~\ref{cor: scalar}.
    Note that since $p$ was self-adjoint, $L$ is also self-adjoint and $L$ can be written as 
    \[
        L = \begin{bmatrix} 1 & uQ^{-1}\\ 0 & 1\end{bmatrix}\begin{bmatrix} p & 0\\ 0 & Q\end{bmatrix}\begin{bmatrix} 1 & 0\\ Q^{-1}v & 1\end{bmatrix}.
    \]
    For convenience, let $A=\begin{bmatrix} 1 & uQ^{-1}\\ 0 & 1\end{bmatrix}$ and $B=\begin{bmatrix} 1 & 0\\ Q^{-1}v & 1\end{bmatrix}$ and let $W=\begin{bmatrix} w & 0 \\ 0 & 0\end{bmatrix}$ then note that $B^{-1}W=W=WA^{-1}$, and $LW=L^\ast W=0$. Thus we have $W\pe(L)W=0$ from Proposition~\ref{prop:zero}.
    Let $\mathbf{s}=(s_1,...,s_n)$ be a $(B, (\delta_{ij}E_B)_{i,j\in I})$-semicircular family free with amalgamation from $B$. Applying $\Psi$ on $W\pe(L)W=0$ gives us
    \begin{align*}
        0&=\Psi\left(W \pe(L) W\right) = W(a_1\otimes s_1 +\cdots +a_n\otimes s_n)W \\
        &= WA^{-1} \left(L(\mathbf{s})-E_B(L(\mathbf{s}))\right)B^{-1}W\\
        &= WA^{-1} L(\mathbf{s}) B^{-1}W - cW \\
        &= W\begin{bmatrix} p(s) & 0 \\ 0 & Q(s) \end{bmatrix}W- cW
    \end{align*}
    where $c \in \C$ is the $(1,1)$-coordinate of $a_0=E_B(L(\mathbf{s}))$.
    Hence, we get $wp(s)w=cw\in M$ on the $(1,1)$-coordinate and by  $E_M(p(s))=\tau(p(s))1_M$ from freeness, 
    \[
        wp(s)w= E_M(wp(s)w)=wE_M(p(s))w= w\tau(p(s))w.
    \]
    Then denoting $\overset{\circ}{p}(s)=p(s)-\tau(p(s))$ we have 
    \[ 0 = \tau\left(w\overset{\circ}{p}(s) w w \overset{\circ}{p}(s)w\right)=\tau\left( w\overset{\circ}{p}(s)\overset{\circ}{p}(s)w\right)\tau(w) = \tau\left(\overset{\circ}{p}(s)^2 w\right)\tau(w) =\tau\left(\overset{\circ}{p}(s)^2\right)\tau(w)^2\]
    thus $\|\overset{\circ}{p}(s)\|_2^2=0$ and $p(s)=\tau(p(s))$. 
    Form \cite[Corollary 1.2]{SK15}, $p$ must be a constant polynomial which contradicts the assumption.   $\hfill\blacksquare$  

\end{rem}

%%%%%%%%%%%%%%%%%%%%%%%%%%%%%%%%%%%%%%
%%%%%                           %%%%%%
%%%%%%%%%%%%%%%%%%%%%%%%%%%%%%%%%%%%%%
\section{Relative diffuseness and the center of $B\vee(B'\cap M)$}\label{section: relative diffuse}
In this section we show that when $\eta_{ii}=E_B$ for each $i\in I$ and $\x$ admits a $(B,\eta)$-conjugate system, $B\vee (B'\cap M)$ is diffuse relative to $B$. Moreover, we show that the center of $B\vee (B'\cap M)$ is actually the center of $B$ if $\eta=(\delta_{ij}E_B)_{i,j\in I}$. To accomplish the former, we will use Popa's intertwining technique. For the latter, we follow Dabrowski's proof from \cite[Theorem 1]{MR2606868}. Before proceeding, we prove following technical proposition, which holds for arbitrary covariant matrix.

%%%%%%%%%%%%%%%%%%%%%%%%%%%%%%%%%%%%%%
%%%%%                           %%%%%%
%%%%%%%%%%%%%%%%%%%%%%%%%%%%%%%%%%%%%%
\begin{prop}\label{prop:central vectors}
    Let $(M,\tau)$ be a tracial von Neumann algebra. Suppose $B\leq N\leq M$ are von Neumann sub-algebras and $N\nprec_M B$. Then $L^2(M\boxtimes_\eta M,\tau)$ has no $N$-central vectors.    
\end{prop}
\begin{proof}
    Suppose $\zeta$ is a $N$-central vector in $L^2(M\boxtimes_\eta M)$. Then for $\epsilon >0$, there exists $\zeta_0=\sum_{i\in F}\sum_{k=1}^{d} x_{i}^{(k)}e_i y_{i}^{(k)} \in M\boxtimes_\eta M$ for some finite set $F\subset I$ such that $\|\zeta-\zeta_0\|_2 <\epsilon$. Since $N\nprec_M B$, there exists a unitary $u\in \mathcal{U}(N)$ so that $\|E_B(xe_i y)\|_\tau < \frac{\epsilon}{d|F|~\|\eta(1)\|^{1/2}}$ for all $x,y\in \{x_i^{(k)}, y_i^{(k)},x_i^{(k)}\ast, y_i^{(k)}\ast:i\in F, 1\leq k\leq d\}$. Noting that $\|\eta_{ij}(x)\|_\tau \leq \|\eta(1)\|\|x\|_\tau$ by \cite[Lemma 3.13]{JLNP25}, one has 
    \begin{align*}
        |\< u\zeta_0u^\ast, \zeta_0 \>| &\leq \sum_{i,j\in F} \sum_{k,l=1}^d |\<ux_{i}^{(k)} e_i y_{i}^{(k)}u^\ast, x_{j}^{(l)} e_j y_{j}^{(l)}\>|= \sum_{i,j\in F} \sum_{k,l=1}^d |\tau\left(u y_{i}^{(k)\ast} \eta_{ij}(x_{i}^{(k)\ast} u^\ast x_{j}^{(l)})y_{j}^{(l)}\right)|\\
        &=\sum_{i,j\in F} \sum_{k,l=1}^d |\tau
        \left(\eta_{ij}\left(x_i^{(k)\ast} u^\ast x_j^{(l)}\right)E_B\left(y_j^{(l)}uy_i^{(k)\ast}\right)\right)| \\
        &\leq  \sum_{i,j\in F}\sum_{k,l=1}^d \| \eta_{ij}\left(E_B\left( x_i^{(k)\ast} u^\ast x_j^{(l)}\right)\right)\|_\tau \| E_B\left(y_j^{(l)}uy_i^{(k)\ast}\right)\|_\tau\\
        &\leq \sum_{i,j\in F}\sum_{k,l=1}^d\|\eta(1)\| \|E_B\left( x_j^{(l)\ast} ux_i^{(k)}\right)^\ast\|_\tau \| E_B\left(y_j^{(l)}uy_i^{(k)\ast}\right)\|_\tau  <\epsilon^2
    \end{align*}
    Hence we have 
    \begin{align*}
        \| \zeta\| _2^2=\< u\zeta u^\ast ,\zeta\> &= | \<u(\zeta-\zeta_0)u^\ast +u\zeta_0 u^\ast , (\zeta-\zeta_0)+\zeta_0\>|\\
        &\leq | \<u(\zeta-\zeta_0)u^\ast, (\zeta-\zeta_0)\>| + |\<u(\zeta-\zeta_0)u^\ast,\zeta_0\>|+|\<u\zeta_0u^\ast,\zeta-\zeta_0\>| +|\<u\zeta_0 u^\ast,\zeta_0\>|\\
        &\leq \| \zeta-\zeta_0\| _2^2 +\| \zeta-\zeta_0\| _2\cdot \| \zeta_0\| _2 +\| \zeta_0\| _2\cdot \| \zeta-\zeta_0\| _2 +|\< u\zeta_0u^\ast, \zeta_0 \>|\\
        &< \epsilon^2 +2\epsilon\| \zeta_0\| _2 +\epsilon^2\\
        &\leq \epsilon^2 + 2\epsilon (\|\zeta_0-\zeta\|_2 +\|\zeta\|_2) +\epsilon^2\\
        &< 4\epsilon^2 +2\epsilon \|\zeta\|_2.
    \end{align*}
    Since $\epsilon$ was arbitrary, $\| \zeta\| _2=0$ and there are no non-trivial $N$-central vectors in $L^2(M\boxtimes_\eta M,\tau)$.
\end{proof}
Next, we must adapt the $\eta$-partial derivative to $B\vee (B'\cap M)$.
%%%%%%%%%%%%%%%%%%%%%%%%%%%%%%%%%%%%%%
%%%%%                           %%%%%%
%%%%%%%%%%%%%%%%%%%%%%%%%%%%%%%%%%%%%%
\begin{prop} Let $(M,\tau)$ be a tracial von Neumann algebra generated by a von Neumann algebra $B$ and a tuple of self-adjoint operators $\x=(x_i)_{i\in I}$. Assume that a $(B,\eta)$-conjugate system exists for $\x$ and a covariance matrix $\eta$. Then $\con(\bx)\subset \dom(\overline{\pe})\cap M$ and consequently $\overline\pe$ is densely defined on $B\vee (B'\cap M)$.
\end{prop}
\begin{proof}
    For $p\in\bx$, let $\xi\in \overline{\operatorname{conv}}^{\|\cdot\|_\tau}\{upu^\ast:u\in \mathcal{U}(B)\}$, which we approximate by a sequence 
    \[p_n=\sum_{i=1}^{d_n} \alpha_{n,i}u_{n,i}pu_{n,i}^\ast \in \operatorname{conv}\{upu^*\colon u\in \mathcal{U}(B)\}\subset \bx\]
    where $d_n\in \N, \alpha_{n,i}>0, \sum_{i=1}^{d_n}\alpha_{n,i}=1$ and $u_{n,i}\in \mathcal{U}(B)$ for $i=1,...,d_n$ for each $n\in\N$.
    Then 
        \[ \|\pe(p_n)\|_2 = \left\|\sum_{i=1}^{d_n} \alpha_{n,i}u_{n,i}\pe(p)u_{n,i}^\ast\right\|_2 \leq \sum_{i=1}^{d_n} \alpha_{n,i}\|u_{n,i}\pe(p)u_{n,i}^\ast \|_2 = \|\pe(p)\|_2. \]
    Thus $(\pe(p_n))_{n\in \N}$ is bounded, so we can find a weak cluster point $\zeta\in L^2(M\boxtimes_\eta M,\tau)$. Reducing to subsequence, we may assume $\pe(p_n)\to \zeta$ weakly. By Mazur's lemma, for all $N\in \N$ there exists $\zeta_N\in \operatorname{conv}\{\pe(p_n):n\geq N\}$ such that $\| \zeta_N-\zeta\| _2 <1/N$. Thus $\zeta_N\to \zeta$. Since $\zeta_N\in \operatorname{conv}\{\pe(p_n):n\geq N\}$, for each $N$ there exists $k_N\in \N$ and $\beta_j\in(0,1]$ for $j=1,...,k_N$ such that $\sum_{j=1}^{k_N}\beta_j =1$, and $\zeta_N =\sum_{j=1}^{k_N} \beta_j\pe(p_{n_j})=\pe(\sum_{j=1}^{k_N} \beta_j p_{n_j})$ for $n_1,...,n_{k_N}\geq N$. Then $\sum_{j=1}^{k_N} \beta_j p_{n_j} \to \xi$ as $N\to \infty$.
    Therefore, $\xi\in\dom(\overline\pe)$ with $\overline\pe(\xi)=\zeta$. Recalling that $\con(x)\in \overline{\operatorname{conv}}^{\|\cdot\|_\tau}\{uxu^\ast:u\in \mathcal{U}(B)\}$ for any $x\in M$ (see, for example, \cite[Lemma 2.2]{MR4852258}), this shows that $\con(\bx) \subset \dom(\overline\pe)\cap M$. The weak* density of $\bx$ in $M$ implies that of $\con(\bx)$ in $B\vee (B'\cap M)$, and so $\overline\pe$ is a densely defined derivation on $B\vee (B'\cap M)$.
\end{proof}

Observe that if $\eta_{ii}=E_B$, then $e_i$ is $B$-central in $L^2(M\boxtimes_\eta M,\tau)$. Thus $\pe(\sum_j \alpha_ju_jx_iu_j^\ast)=\sum_j \alpha_ju_je_iu_j^*= e_i$ and so $\overline\pe(\con(x_i))=e_i$. 

%%%%%%%%%%%%%%%%%%%%%%%%%%%%%%%%%%%%%%
%%%%%                           %%%%%%
%%%%%%%%%%%%%%%%%%%%%%%%%%%%%%%%%%%%%%
We are grateful to Ionut Chifan for suggesting the proof of part (c) of the following proposition.
\begin{prop}\label{prop:xit}
     Let $(M,\tau)$ be a tracial von Neumann algebra generated by a von Neumann algebra $B$ and a tuple of self-adjoint operators $\x=(x_i)_{i\in I}$. Assume that an $i$-th $(B,\eta)$-conjugate variable exists for at least one $i\in I$ and a covariance matrix $\eta$ with $\eta_{ii}=E_B$. Then:  
    \begin{enumerate}[label=(\alph*)]
        \item\label{prop:xit no atom} $\con(x_i)$ has no atoms in $B$;
        
        \item\label{prop:xit rel.diff} $B\<\con(x_i)\>''$ is diffuse relative to $B$;

        \item\label{prop:normalizer} $B\<\con(x_i)\>'' \nprec_{\mathcal{N}_M(B)''} B$

        \item\label{prop:xit corner} $B\<E_{B'\cap M}(x_i)\>'' \nprec_{B\vee (B'\cap M)} B$
        
        \item $B\vee(B'\cap M)$ is diffuse relative to $B$.
    \end{enumerate}
\end{prop}
\begin{proof} 
Let $\xit=\con(x_i), \ait=B\<\xit\>''$ and $\at=B\vee(B'\cap M)$.\\
\textbf{(a):} Since $\overline{\pe}(\xit)=e_i$, arguing as in the Theorem~\ref{thm:no atom} shows that $\xit$ has no atoms in $B$.\\

\noindent\textbf{(b):} Suppose towards a contradiction, $\ait\prec_{\ait} B$. Then by Theorem~\ref{pre:popathm}, there exist projections $p\in \ait ,q\in B$, a non-zero partial isometry $v\in q\ait p$, and normal $*$-homomorphism $\theta:p\ait p\to qBq$ such that $\theta(pap)v=v(pap)$ for all $a\in \ait $. Then for $a=\xit$
        \[ \theta(p\xit p)vv^\ast =v(p\xit p)v^\ast= v\xit v^\ast = \xit vv^\ast,\]
where the last equality follows from $\xit$ being in the center of $\ait$. This implies that $\theta(p\xit p)\in B$ is an atom for $\xit$, contradicting \ref{prop:xit no atom}. Therefore $\ait$ is diffuse relative to $B$. \\

\noindent\textbf{(c):} We will show a more general case where $A,B,P\leq M$ be von Neumann subalgebras such that $B\subset P$ and $A\subset P\subset \mathcal{N}_M(B)'' \subset M$, then $A\nprec_P B$ implies $A\nprec_{\mathcal{N}_M(B)''} B$. Indeed,
since $A\nprec_P B$, there exists a net of unitaries $(u_k)_{k}\subset A$ that satisfies $\lim_{k\to\infty}\|E_B(xu_ky)\|_\tau=0$ for all $x,y\in P$. Let $a,b\in \mathcal{N}_M(B)$. Then 
    \begin{align*}
        \|E_B(au_kb)\|_\tau&=\|a^\ast E_B(au_kbaa^\ast)a\|_\tau = \|E_{a^\ast B a}(u_kba)\|_\tau = \|E_B(u_kba)\|_\tau \\
        &=\|E_B\circ E_P(u_kba)\|_\tau=\|E_B(u_kE_P(ba))\|_\tau
    \end{align*}
    converges to $0$ as $k\to\infty$. 
    For elements in $\mathcal{N}_M(B)''$ it follows by a standard approximation argument. 
    % For $a,b\in \mathcal{N}_M(B)''$ and $\epsilon >0$, there exist $a_0,b_0 \in \mathcal{N}_M(B)$ with $\|a_0\|\leq \|a\|, \|b_0\|\leq \|b\|$ and $\|a-a_0\|,\|b-b_0\|<\epsilon$. Then 
    % \begin{align*}
    %     \|E_B(au_kb)\|_2&=\|E_B(a-a_0+a_0)u_k(b-b_0+b_0)\|_2 \\
    %     &\leq\|E_B(a_0u_kb_0)\|_2+\|E_B((a-a_0)u_kb_0)\|_2+\|E_B(a_0u_k(b-b_0))\|_2+\|E_B((a-a_0)u_k(b-b_0))\|_2.
    % \end{align*}
    % Note that 
    %     \[\|E_B((a-a_0)u_kb_0)\|_2 \leq\|(a-a_0)u_kb_0\|_2\leq \|a-a_0\|_2\cdot\|u_kb_0\| \leq \epsilon \|b\|.\] 
    % Similarly, $\|E_B(a_iu_k(b-b_0))\|_2\leq \epsilon\|a\|$ and $\|E_B((a-a_0)u_k(b-b_0))\|_2 \leq 2\epsilon\|b\| $. Hence    
    %     \[\lim_{k\to\infty}\|E_B(au_kb)\|_2\leq \limsup_{k\to\infty}\|E_B(a_0u_kb_0)\|_2+4\epsilon\max(\|a\|,\|b\|).\]
    % Therefore $\|E_B(au_kb)\|_2$ converges to $0$ as $k\to\infty$.
    % Thus we have $A\nprec_{\mathcal{N}_M(B)''} B$.
    By setting $A=P=\ait$ from above, it is immediate from \ref{prop:xit rel.diff}.\\

\noindent \textbf{(d):} Since $\mathcal{N}_M(B)''\supset B\vee (B'\cap M)$, it is immediate from \ref{prop:normalizer}.\\

\noindent \textbf{(e):} Suppose $\at\prec_{\at} B$. Then by the Lemma~\ref{pre:lem1}, $\ait\prec_{\ait}B$ which contradicts \ref{prop:xit rel.diff}.

\end{proof}

%%%%%%%%%%%%%%%%%%%%%%%%%%%%%%%%%%%%%%
%%%%%                           %%%%%%
%%%%%%%%%%%%%%%%%%%%%%%%%%%%%%%%%%%%%%

\begin{thm}[Theorem~\ref{thm a}]\label{thm:main}
    Let $(M,\tau)$ be a tracial von Neumann algebra generated by a von Neumann algebra $B$ and a tuple of self-adjoint operators $\x=(x_i)_{i\in I}$ with $|I|>1$. Assume that a $(B,\eta)$-conjugate system exists for $\x$ and a covariance matrix $\eta=(\delta_{ij}E_B)_{i,j\in I}$. Then $Z(B\vee(B'\cap M))=Z(B)$. In particular, $Z(M)\subset Z(B)$.
\end{thm}
\begin{proof}
    Let $\xit=\con(x_i), \ait=B\<\xit\>''$ and $\at=B\vee(B'\cap M)$. First, observe that there are no $\ait$-central vectors in $L^2(\at \boxtimes_{E_B}\at,\tau)$ by Proposition~\ref{prop:xit}~\ref{prop:xit corner} and Proposition~\ref{prop:central vectors}.
    From here, we follow the argument used in the proof of \cite[Theorem 1]{MR2606868}. 
    Note that if $\eta_{ij}=0$ for $i\neq j$, one has
    \[ L^2(M\underset{\eta}{\boxtimes}  M,\tau)=\bigoplus_{j\in I}L^2(M\underset{\eta_{jj}}{\boxtimes} M,\tau) \]
    and 
    $\pe$ decomposes as $\sum_{j\in I} \pej$ where $\pej\colon \bx\to L^2(M\boxtimes_{\eta_{jj}} M,\tau)$ satisfies
    \begin{align*}
        \pej(x_i)= \delta_{ij} e_j \qquad i\in I,
    \end{align*}
    and $\pej(b)=0$ for all $b\in B$.
   
    Consider the positive unbounded operator on $L^2(M,\tau)$ 
        \[\Delta_j:=\peaj\overline{\pej}\] 
    and a bounded operator on $L^2(M,\tau)$ for all $t>0$
        \[\zeta_{t,j} :=\left( \frac{t}{t+\Delta_j}\right)^{1/2}.\] 
    Then $\| \zeta_{t,j}(y)-y\| _2 \to 0$ as $t\to\infty$ for $y\in L^2(M,\tau)$ and since $\zeta_{t,j}(L^2(M))\subset \dom(\Delta_j)\subset\dom(\overline{\pej})$, $\overline{\pej}\circ\zeta_{t,j}$ is bounded (see \cite[Section 2]{MR2470111} or \cite[Section 1]{MR2606868}). Note that $\Delta_j(\xit)=0$ and $\xit$ is in the domain of $\Delta_j$ for $i\neq j$. For $\xi\in \dom(\Delta_j)\subset \dom(\overline{\pej})$, $i\neq j$ and $p\in \bx$ we have
        \begin{align*}
            \< \Delta_j(\xit\xi), p\>_\tau &= \< \overline{\pej}(\xit\xi),\pej(p)\> \\
            &= \< \xit \overline{\pej}(\xi), \pej(p)\>\\
            &= \< \overline{\pej}(\xi), \xit\pej(p) \>\\
            &= \< \overline{\pej}(\xi), \overline{\pej}(\xit p)\>\\
            &= \< \Delta_j(\xi),  \xit p\>_\tau = \< \xit\Delta_j(\xi),  p\>_\tau.
        \end{align*}
    Thus $\Delta_j(\xit\xi) = \xit\Delta_j(\xi)$. Likewise, $\Delta_j(\xi \xit)=\Delta_j(\xi)\xit$.
 
    Now, for arbitrary $\xi\in L^2(M,\tau)$ let $f=\zeta_{t,j}^2(\xit \xi)=\frac{t}{t+\Delta_j}(\xit \xi)$ and $g=\zeta_{t,j}^2(\xi)=\frac{t}{t+\Delta_j}(\xi)$. Then 
    \begin{align*}
        \frac{1}{t}(t+\Delta_j)(f-\xit g)&=\xit \xi-\xit g- \frac{1}{t}\Delta_j(\xit g)\\
        &=\xit \xi-\xit g-\xit\frac{1}{t}\Delta_j(g)=\xit \xi-\xit((1+\frac{\Delta_j}{t})(g))=\xit \xi-\xit \xi=0.
    \end{align*}
    Thus $\zeta_{t,j}^2(\xit \xi)=\xit \zeta_{t,j}^2(\xi)$ and the functional calculus implies $\zeta_{t,j}(\xit \xi)=\xit\zeta_{t,j}(\xi)$. Similarly, $\zeta_{t,j}(\xi \xit )=\zeta_{t,j}(\xi)\xit$. Therefore $\zeta_{t,j}([\xit,\xi])=[\xit,\zeta_{t,j}(\xi)]$.
    Assume $z\in \at'\cap \at$. Then
        \[0=\overline{\pej}\circ \zeta_{t,j}([\xit,z])=[\xit,\overline{\pej}\circ\zeta_{t,j}(z)].\]  Similarly for $b\in B$ one has
        \[0=\overline{\pej}\circ\zeta_{t,j}([b,z])=[b,\overline{\pej}\circ\zeta_{t,j}(z)].\] 
    Therefore $\overline{\pej}\circ\zeta_{t,j}(z)$ is $\ait$ central. Since there are no non-zero $\ait$-central vectors, $\overline{\pej}\circ\zeta_{t,j}(z)=0$ for all $t>0$. So
     $\zeta_{t,j}(z)\to z$ and $\overline{\pej}(\zeta_{t,j}(z))=0$ imply $z\in \dom(\overline{\pej})$ with $\overline{\pej}(z)=0$. But then $0=\overline{\pej}([x_j, z])=[e_j,z]$, and so
    \begin{align*}
        0&=\< [e_j,z],[e_j,z]\> =\tau(z^\ast z-z^\ast E_B(z)-E_B(z^\ast)z+E_B(z^\ast z) )\\&\geq \tau(z^\ast z-z^\ast E_B(z)-E_B(z^\ast)z+E_B(z^\ast)E_B( z) )=\| z-E_B(z)\| _\tau^2.
    \end{align*} 
    Therefore $z=E_B(z)\in B$ and $\at'\cap \at \subset B'\cap B$. The reverse inclusion is immediate from $\at=B\vee (B'\cap M)$. For the second part, let $a\in Z(M)$. Then $a$ commutes with $B$ hence $a\in B\vee(B'\cap M)$. Also, $a\in M'\subset (B\vee(B'\cap M))'$. Therefore $a\in Z(B\vee(B'\cap M))=Z(B)$.  
\end{proof}

\begin{cor}\label{cor:irred}
    Let $(M,\tau)$ be a tracial von Neumann algebra generated by a von Neumann algebra $B$ and a tuple of self-adjoint operators $\x=(x_i)_{i\in I}$ with $|I|>1$. Assume that a $(B,\eta)$-conjugate system exists for $\x$ and a covariance matrix $\eta=(\delta_{ij}E_B)_{i,j\in I}$. If $B$ is a factor, then every intermediate algebra $B\vee(B'\cap M)\leq N \leq M$ is irreducible in $M$.
\end{cor}
\begin{proof}
    Let $N$ be such that $B\vee(B'\cap M)\leq N \leq M$. Then $B\leq N$ and hence  
        \[N'\cap M\leq B'\cap M\leq B\vee(B'\cap M).\] 
    On the other hand, $B\vee (B'\cap M)\leq N$ implies \[N'\cap M\leq (B\vee(B'\cap M))'\cap M.\]
    Therefore $N'\cap M\leq Z(B\vee (B'\cap M))=Z(B)=\C$.
\end{proof}
% In the case when $x_i\in B'\cap M$ for all $i\in I$ and $M=\bx''$, one has $B\vee (B'\cap M) = M$. We therefore have the following corollary.

% \begin{cor}\label{cor:center}
%     Let $(M,\tau)$ be a tracial von Neumann algebra generated by a von Neumann algebra $B$ and self-adjoint operators $\{x_i\}_{i\in I}$. Assume that a $(B,\eta)$-conjugate system exists for $\{x_i\}_{i\in I}$ and a covariance matrix $\eta=(\delta_{ij}E_B)_{i,j\in I}$. If $x_i$ commutes with $B$ for all $i\in I$ then $Z(M)=Z(B)$.
% \end{cor}
\begin{ex}
    For $t>0$, let $(x_i(t))_{i\in I}$ and $N_t$ be as in Example~\ref{ex:amalgamation}. If $\eta=\delta_{ij}E_B$, we have $Z(N_t)\subset Z(B)$ by the Theorem~\ref{thm:main}. In particular, $N_t$ is a factor whenever $B$ is a factor.
\end{ex}
We conclude by noting that Proposition~\ref{prop:central vectors} allows us to give a relative version of \cite[Theorem 4]{MR2606868} (see Remark~\ref{rem:dab4}).
\begin{prop}\label{prop:dab4}
    Let $(M,\tau)$ be a tracial von Neumann algebra generated by a von Neumann algebra $B$ and a tuple of self-adjoint operators $\x=(x_i)_{i\in I}$. Suppose an intermediate von Neumann algebra $B\leq N \leq M$ satisfies $N\nprec_M B$. Further suppose that the $i_0$-th $(N,\eta\circ E_B)$-conjugate variable exists for some $i_0\in I$, where $\eta=(\eta_{ij})_{i,j\in I}$ is a covariance matrix over $B$ satisfying $\eta_{i_0i_0}(1)=1$. Then 
    \[ 
        (N\cup \{x_{i_0}\})'\cap M \subset Z(B).
    \]
\end{prop}
\begin{proof}
    As usual, we will abuse notation slightly to identify $\eta\circ E_B$ with $\eta$. Fix $z\in (N\cup \{x_{i_0}\})'\cap M$. Let $\partial_{\eta,i_0}:N\<\x\>\to L^2(M\boxtimes_\eta M)$ be the usual derivations satisfying $\partial_{\eta,i_0}(x_j)=\delta_{i_0j}e_i$ and $\partial_{\eta,i_0}(y)=0$ for all $y\in N$, and let $\zeta_{t,i_0}$ as in the proof of Theorem~\ref{thm:main}. Then for $y\in N$, arguing as in the proof of Theorem~\ref{thm:main} we have 
    \[
        0=\partial_{\eta,i}\circ\zeta_{t,i_0}([y,z]) = [y,\partial_{\eta,i}\circ\zeta_{t,i_0}(z)].
    \]
    There are no $N$-central vectors by Proposition~\ref{prop:central vectors}, hence $\partial_{\eta,i}\circ\zeta_{t,i_0}(z)$ must be zero. As $\|\zeta_{t,i_0}(z) - z\|_\tau\to 0$, it follows that $z\in \dom(\overline{\partial_{\eta,i_0}})$ with $\overline{\partial_{\eta,i}}(z)=0$.
    Thus $0=\overline{\partial_{\eta,i_0}}([x_{i_0},z])=[e_{i_0},z]$.
    
    Noting that $\|\eta_{i_0i_0}(x)\|_\tau \leq \|x\|_\tau$ by \cite[Lemma 3.13]{JLNP25}.
    we have 
    \begin{align*}
        \< \eta_{i_0i_0}(x),x\> &= \< \eta_{i_0i_0}(E_B(x)),E_B(x)\>\\
        &\leq | \< \eta_{i_0i_0}(E_B(x)),E_B(x)\> |\\
        &\leq \| \eta_{i_0i_0}(E_B(x))\|_\tau \|E_B(x)\|_\tau \\
        &\leq \|E_B(x)\|_\tau^2\\
        &=\<E_B(x),x\>_\tau
    \end{align*}
    Therefore, 
    \begin{align*}
        0&=\<[e_{i_0},z],[e_{i_0},z]\>=\tau\left(z^\ast z-z^\ast\eta_{i_0i_0}(z)-\eta_{i_0i_0}(z^\ast)z+\eta_{i_0i_0}(z^\ast z)\right)\\
        &=2\tau(z^\ast z-\eta_{i_0i_0}(z)^\ast z) = 2\<z-\eta_{i_0i_0}(z),z\>\\
        &\geq 2\<z-E_B(z),z\>=\|z-E_B(z)\|_\tau^2
    \end{align*}
    Thus $z=E_B(z)$ and $z\in B$. Also, since $z\in (N\cap \{x_i\})' \subset B'$, $z\in Z(B)$.
\end{proof}
\begin{rem}\label{rem:dab4}
    For $B=\C$ and $\eta=(\delta_{ij}\tau)_{i,j\in I}$, Proposition~\ref{prop:dab4} recovers \cite[Theorem 4]{MR2606868} since $N\nprec_M \C$ is equivalent to $N\nprec_N\C$ ($N$ is diffuse). Indeed, if $N$ is diffuse, there exists a net of unitaries converging weakly to zero. Since weak convergence holds under any representation of $N$, they will still converge weakly to zero on $L^2(M)$. For the converse, assume $N\prec_N \C$. Using Theorem~\ref{pre:popathm}(\ref{popa 3}), there exists $N$-$\C$-correspondence $K\leq L^2(N)$ with finite dimension. Then $K\leq L^2(M)$ still with finite dimension, hence $N\prec_M \C$.
\end{rem}
%%%%%%%%%%%%%%%%%%%%%%%%%
%       References      %
%%%%%%%%%%%%%%%%%%%%%%%%%

\bibliographystyle{amsalpha}
\bibliography{reference}

\end{document}